\documentclass[11pt, reqno]{amsart}
\allowdisplaybreaks[3]
\usepackage[left=3cm,right=3cm,top=2cm,bottom=2cm]{geometry}

\usepackage{amsthm}
\usepackage{amscd}
\usepackage{amsmath}
\usepackage{amsfonts}
\usepackage{amssymb}
\usepackage{graphicx}
\usepackage{amsopn}
\usepackage[usenames]{color}
\usepackage{tikz}
\usetikzlibrary{arrows,shapes,snakes,automata,backgrounds,petri}
\usepackage{textcomp}
\usepackage[all]{xy}

\theoremstyle{theorem}
\newtheorem{theorem}{Theorem}
\newtheorem{corollary}[theorem]{Corollary}
\newtheorem{lemma}[theorem]{Lemma}
\newtheorem{prop}[theorem]{Proposition}

\theoremstyle{definition}
\newtheorem{definition}[theorem]{Definition}

\newtheorem{fact}[theorem]{Fact}
\newtheorem{remark}[theorem]{Remark}
\newtheorem{Question}[theorem]{Question}
\newtheorem{construction}[theorem]{Construction}

\def\D{\mathcal{D}}
\def\M{\mathcal{M}}
\def\R{\mathcal{R}}
\def\X{\mathcal{X}}
\def\C{\mathcal{C}}
\def\A{\mathcal{A}}
\def\I{\mathcal{I}}
\def\K{\mathcal{K}}
\def\V{\mathcal{V}}
\def\P{\mathcal{P}}
\def\N{\mathcal{N}}
\def\T{\mathcal{T}}
\def\U{{\rm U}}

\def\B{\mathcal{B}}
\def\S{\Sigma}
\def\Mod{{\rm Mod}}

\def\Ind{{\rm Ind}}

\def\PMod{{\rm PMod}}
\def\Stab{{\rm Stab}}
\def\Orb{{\rm Orb}}
\def\PB{{\rm PB}}

\def\id{{\rm id}}
\def\Z{\mathbb{Z}}
\def\Q{\mathbb{Q}}

\def\Surger{{\rm Surger}}
\def\Drain{{\rm Drain}}
\def\Homeo{{\rm Homeo}}
\def\H{{\rm H}}
\def\Sp{{\rm Sp}}
\def\SL{{\rm SL}}

\def\U{{\rm U}}

\newcommand{\rk}{\mathop{\mathrm{rk}}}
\newcommand{\cd}{\mathop{\mathrm{cd}}}
\newcommand{\Mat}{\mathop{\mathrm{Mat}}}
\newcommand{\sign}{\mathop{\mathrm{sign}}}

\numberwithin{theorem}{section}

\begin{document}

\title{On the structure of the top homology group of the Johnson kernel}
\thanks{I was partially supported by the HSE University Basic Research Program and by the Simons Foundation. \smallskip }

\address{National Research University Higher School of Economics, Russian Federation}
\address{Skolkovo Institute of Science and Technology, Skolkovo, Russia}
\email{spiridonovia@ya.ru}
\author{Igor A. Spiridonov}


\thispagestyle{empty}

\vspace{0.5cm}

\begin{abstract}
	
The Johnson kernel is the subgroup $\mathcal{K}_g$ of the mapping class group ${\rm Mod}(\Sigma_{g})$ of a genus $g$ oriented closed surface $\Sigma_{g}$ generated by all Dehn twists about separating curves. In this paper we study the structure of the top homology group ${\rm H}_{2g-3}(\mathcal{K}_g, \mathbb{Z})$. For any collection of $2g-3$ disjoint separating curves on $\Sigma_{g}$ one can construct the corresponding abelian cycle in the group ${\rm H}_{2g-3}(\mathcal{K}_g, \mathbb{Z})$; such abelian cycles will be called simple. In this paper we describe the structure of $\mathbb{Z}[{\rm Mod}(\Sigma_{g})/ \mathcal{K}_g]$-module on the subgroup of ${\rm H}_{2g-3}(\mathcal{K}_g, \mathbb{Z})$ generated by all simple abelian cycles and find all relations between them.
\end{abstract}
\maketitle

\section{Introduction}

Let $\S_{g}$ be a compact oriented genus $g$ surface. Let $\Mod(\S_{g}) = \pi_{0}(\Homeo^{+}(\S_{g}))$ be the \textit{mapping class group} of $\S_{g}$, where $\Homeo^{+}(\S_{g})$ is the group of orientation-preserving homeomorphisms of $\S_{g}$. The group $\Mod(\S_{g})$ acts on $\H = \H_1(\S_g, \Z)$. This action preserves the algebraic intersection form, so we have the representation $\Mod(\S_{g}) \rightarrow \Sp(2g, \Z)$, which is well-known to be surjective. The kernel $\I_g$ of this representation is known as the \textit{Torelli group}. This can be written as the short exact sequence 
\begin{equation} \label{exact1}
1 \rightarrow \I_g \rightarrow \Mod(\S_{g}) \rightarrow \Sp(2g, \Z) \rightarrow 1.
\end{equation}

The \textit{Johnson kernel} $\K_g$ is the subgroup of $\I_g$ generated by all Dehn twists about separating curves. Johnson \cite{Johnson2} proved that the group $\K_g$ also can be defined as the kernel of the surjective \textit{Johnson homomorphism} $\tau: \I_g \rightarrow \wedge^{3} \H / \H$,  where the inclusion $\H \hookrightarrow \wedge^{3} \H$ is given by $x \mapsto x \wedge \Omega$ and $\Omega \in \wedge^2\H$ is the inverse tensor of the algebraic intersection form. Therefore, we have the short exact sequence
\begin{equation} \label{exact2}
1 \rightarrow \K_g \rightarrow \I_g \rightarrow \wedge^{3} \H / \H \rightarrow 1.
\end{equation}

Denote by $\mathcal{G}_g$ the quotient group $\Mod(\S_g) / \K_g$. The exact sequences (\ref{exact1}) and (\ref{exact2}) imply that $\mathcal{G}_g$ can be presented as the following extension 
$$
1 \rightarrow \wedge^{3} \H / \H \rightarrow \mathcal{G}_g \rightarrow \Sp(2g, \Z) \rightarrow 1
$$
of the symplectic group by the free abelian group $\wedge^{3} \H / \H$. The groups $\H_*(\K_g, \Z)$ has the natural structure of $\mathcal{G}_g$-module.

In the case $g=1$ the representation $\Mod(\S_{1}) \rightarrow \Sp(2, \Z) = \SL(2, \Z)$ is an isomorphism, so the group $\I_1$ is trivial. Mess \cite{Mess} proved that the group $\I_2 = \K_2$ is a free group with countable number of generators. Therefore, below we assume that $g \geq 3$ unless explicitly stated otherwise.

A natural problem is to study the homology of the group $\K_g$ for $g \geq 3$. The rational homology group $\H_{1}(\K_g, \Q)$ was shown to be finitely generated for $g \geq 4$ by Dimca and Papadima \cite{Dimca2}. This group was computed explicitly for $g \geq 6$ by Morita, Sakasai, and Suzuki \cite{Morita} using the description due to Dimca, Hain, and Papadima \cite{Dimca1}. Recently Ershov and Sue He \cite{Ershov} proved that $\K_g$ is finitely generated in the case 
$g \geq 12$.
This result was extended to any genus $g \geq 4$ by Church, Ershov, and Putman \cite{Church2}. This implies that the group $\H_1(\K_g, \Z)$ is finitely generated, provided that $g \geq 4$.
It is still unknown whether $\K_3$ and $\H_1(\K_3, \Z)$ are finitely generated.

Bestvina, Bux, and Margalit \cite{Bestvina} computed the cohomological dimension of the Johnson kernel $\cd(\K_g) = 2g-3$. Gaifullin \cite{Gaifullin_J} proved that the top homology group $\H_{2g-3}(\K_g, \Z)$ contains a free $\Z[\wedge^{3} \H / \H]$-module of infinite rank. In particular, the group $\H_{2g-3}(\K_g, \Z)$ is not finitely generated.

Recall that for $n$ pairwise commuting elements $h_{1}, \dots, h_{n}$ of the group $G$ one can construct the \textit{abelian cycle} $\A(h_{1}, \dots, h_{n}) \in \H_{n}(G, \Z)$ defined as follows. Consider the homomorphism $\phi: \Z^{n} \rightarrow G$ that maps the generator of the $i^{\rm th}$ factor to $h_{i}$. Then $\A(h_{1}, \dots, h_{n}) = \phi_{*}(\mu_{n})$, where $\mu_{n}$ is the standard generator of $\H_{n}(\Z^{n}, \Z)$.

By a \textit{curve} we always mean an essential simple closed curve on $\S_{g}$. By an (oriented) \textit{multicurve} we mean a finite union of pairwise disjoint and nonisotopic (oriented) curves on $\S_{g}$. An \textit{ordered multicurve} is a multicurve with a fixed order on its components. Usually we will not distinguish between a curve or a multicurve and its isotopy class.
We denote by $T_{\gamma}$ the left Dehn twist about a curve $\gamma$.

\begin{definition}
    An \textit{S-multicurve} is an ordered multicurve consisting of $2g-3$ separating components.
\end{definition}

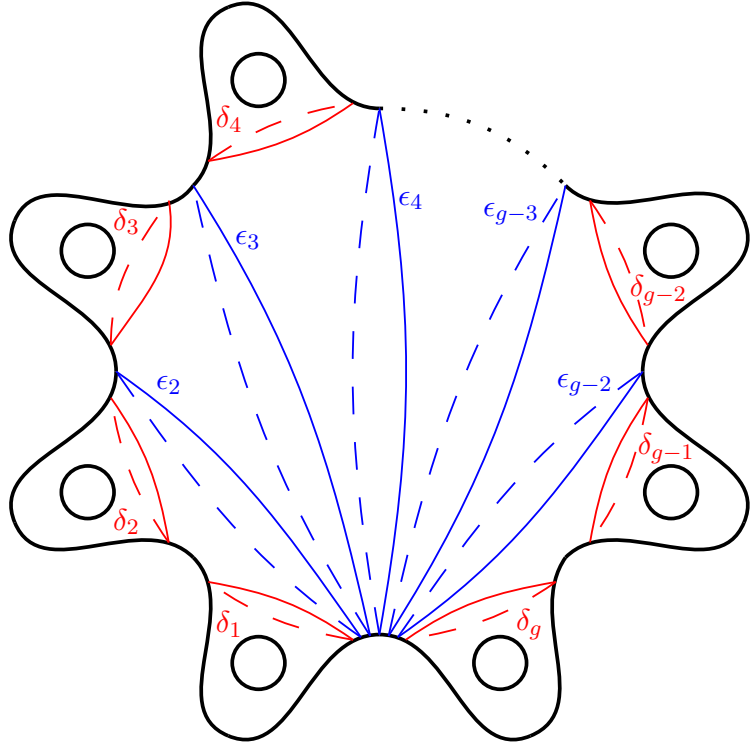
\begin{figure}[h]
	\scalebox{3.5}{
		\begin{tikzpicture}
	
		\draw[] (0, -1) to [out=180, in=337.5] (-0.574, -1.385);
		\draw[] (-0.574, -1.385) to [out=157.5, in=315] (-0.707, -0.707);

		\draw[] (-0.707, -0.707) to [out=135, in=292.5] (-1.385, -0.574);
		\draw[] (-1.385, -0.574) to [out=112.5, in=270] (-1, 0);

		\draw[] (-1, 0) to [out=90, in=247.5] (-1.385, 0.574);
		\draw[] (-1.385, 0.574) to [out=67.5, in=235] (-0.707, 0.707);

		\draw[] (-0.707, 0.707) to [out=45, in=202.5] (-0.574, 1.385);
		\draw[] (-0.574, 1.385) to [out=22.5, in=180] (0, 1);

		\draw[dotted][] (0, 1) to [out=0, in=135] (0.707, 0.707);

		\draw[] (0.707, 0.707) to [out=315, in=112.5] (1.385, 0.574);
		\draw[] (1.385, 0.574) to [out=292.5, in=90] (1, 0);

		\draw[] (1, 0) to [out=270, in=67.5] (1.385, -0.574);
		\draw[] (1.385, -0.574) to [out=247.5, in=45] (0.707, -0.707);

		\draw[] (0.707, -0.707) to [out=235, in=22.5] (0.574, -1.385);
		\draw[] (0.574, -1.385) to [out=202.5, in=0] (0, -1);

		\draw[] (1.108, 0.459) circle (0.1);
		\draw[] (-1.108, -0.459) circle (0.1);
		\draw[] (-1.108, 0.459) circle (0.1);
		\draw[] (1.108, -0.459) circle (0.1);

		\draw[] (-0.459, -1.108) circle (0.1);
		\draw[] (-0.459, 1.108) circle (0.1);
		\draw[] (0.459, -1.108) circle (0.1);

		\draw[red][very thin] (-0.1, -1.02) to [out=145, in=350] (-0.65, -0.8);
		\draw[red][very thin][dashed] (-0.1, -1.02) to [out=170, in=325] (-0.65, -0.8);
		
		\draw[red][very thin] (-0.8, -0.65) to [out=100, in=305] (-1.02, -0.1);
		\draw[red][very thin][dashed] (-0.8, -0.65) to [out=125, in=280] (-1.02, -0.1);
		
		\draw[red][very thin] (-1.02, 0.1) to [out=55, in=280] (-0.8, 0.65);
		\draw[red][very thin][dashed] (-1.02, 0.1) to [out=90, in=235] (-0.8, 0.65);
		
		\draw[red][very thin] (-0.65, 0.8) to [out=10, in=215] (-0.1, 1.02);
		\draw[red][very thin][dashed] (-0.65, 0.8) to [out=35, in=190] (-0.1, 1.02);

		\draw[red][very thin] (0.8, 0.65) to [out=280, in=125] (1.02, 0.1);
		\draw[red][very thin][dashed] (0.8, 0.65) to [out=305, in=100] (1.02, 0.1);
		
		\draw[red][very thin] (1.02, -0.1) to [out=235, in=80] (0.8, -0.65);
		\draw[red][very thin][dashed] (1.02, -0.1) to [out=260, in=55] (0.8, -0.65);
		
		\draw[red][very thin] (0.67, -0.8) to [out=190, in=35] (0.1, -1.02);
		\draw[red][very thin][dashed] (0.67, -0.8) to [out=215, in=10] (0.1, -1.02);

		\draw[blue][very thin] (-0.07, -1.01) to [out=125, in=325] (-1, 0);
		
		\draw[blue][very thin] (-0.035, -1.005) to [out=102.5, in=302.5] (-0.707, 0.707);
		
		\draw[blue][very thin] (0, -1) to [out=80, in=280] (0, 1);
		
		\draw[blue][very thin] (0.035, -1.005) to [out=57.5, in=257.5] (0.707, 0.707);
		
		\draw[blue][very thin] (0.07, -1.01) to [out=35, in=235] (1, 0);

		\draw[blue][very thin][dashed] (-0.07, -1.01) to [out=145, in=305] (-1, 0);
		
		\draw[blue][very thin][dashed] (-0.035, -1.005) to [out=122.5, in=282.5] (-0.707, 0.707);
		
		\draw[blue][very thin][dashed] (0, -1) to [out=100, in=260] (0, 1);

		\draw[blue][very thin][dashed] (0.035, -1.005) to [out=77.5, in=237.5] (0.707, 0.707);
		
		\draw[blue][very thin][dashed] (0.07, -1.01) to [out=55, in=215] (1, 0);
		
		\node[blue][scale=0.3] at (-0.8, -0.05) {$\epsilon_{2}$};
		\node[blue][scale=0.3] at (-0.5, 0.5) {$\epsilon_{3}$};
		\node[blue][scale=0.3] at (0.12, 0.65) {$\epsilon_{4}$};
		\node[blue][scale=0.3] at (0.5, 0.6) {$\epsilon_{g-3}$};
		\node[blue][scale=0.3] at (0.78, -0.05) {$\epsilon_{g-2}$};
		
		\node[red][scale=0.3] at (-0.57, -0.96) {$\delta_{1}$};
		\node[red][scale=0.3] at (-0.96, -0.57) {$\delta_{2}$};
		\node[red][scale=0.3] at (-0.96, 0.57) {$\delta_{3}$};
		\node[red][scale=0.3] at (-0.57, 0.96) {$\delta_{4}$};
		
		\node[red][scale=0.3] at (1.06, 0.31) {$\delta_{g-2}$};
		\node[red][scale=0.3] at (1.09, -0.29) {$\delta_{g-1}$};
		\node[red][scale=0.3] at (0.57, -0.96) {$\delta_{g}$};
		
		\end{tikzpicture}}
	\caption{Surface $\S_{g}$.}
	\label{S-curve_1}
\end{figure}

For example, the multicurve $\delta_{1} \cup \dots \cup \delta_{g} \cup \epsilon_{2} \cup \dots \cup \epsilon_{g-2}$ in Fig. \ref{S-curve_1} is an S-multicurve.
To an S-multicurve $M = \gamma_{1} \cup \dots \cup \gamma_{2g-3}$ we assign the abelian cycle $\A(M) = \A(T_{\gamma_{1}}, \dots, T_{\gamma_{2g-3}}) \in \H_{2g-3}(\K_g, \Z)$. If an S-multicurve $M'$ is obtained from $M$ by a permutation $\pi$ of its components, then $\A(M') = (\sign \pi) \A(M)$.
An easy computation of the Euler characteristic implies that any S-multicurve separates $\S_g$ into $g$ one-punctured tori and $g-3$ three punctured spheres. Throughout the paper we assume that the components of any S-multicurve are ordered such that the curves with numbers $1, \dots, g$ bound one-punctured tori from $\S_g$.

Abelian cycles of the form $\A(M) \in \H_{2g-3}(\K_g, \Z)$ 
for some S-multicurve $M$ will be called \textit{simple abelian cycles}.
Denote by $\A_g \subseteq \H_{2g-3}(\K_g, \Z)$ the subgroup generated by all simple abelian cycles. The author does not know the answer to the following question, which is an interesting problem itself. 

\begin{Question}
	Is the inclusion  $\A_g \subseteq \H_{2g-3}(\K_g, \Z)$ strict for some $g \geq 3$?
\end{Question}

The natural problem is to study the structure of the group $\A_g$. 
Let us identify $\S_{g}$ with the surface shown in the Fig. \ref{S-curve_1} and fix the multicurve $\Delta = \delta_{1} \cup \dots \cup \delta_{g} \cup \epsilon_{2} \cup \dots \cup \epsilon_{g-2}$ on $\S_{g}$. Denote by $\P_g \subseteq \A_g$ the subgroup generated by all simple abelian cycles $\A(M)$, where $M = \delta_1 \cup  \dots \cup \delta_g \cup \epsilon'_2 \cup \dots \cup \epsilon'_{g-2}$ for some separating curves $\epsilon'_2, \dots, \epsilon'_{g-2}$.

By an \textit{unordered symplectic splitting} we mean an orthogonal (w.r.t. the intersection form) decomposition of $\H$ into a direct sum of $g$ subgroups of rank $2$.
Denote $N=\delta_1 \cup \dots \cup \delta_g$ and for each $i$ let $X_i$ be the one-punctured tori bounded by $\delta_i$. Denote $V_i = \H_{1}(X_i, \Z) \subset \H$. Consider the corresponding unordered symplectic splitting $\H = \oplus_i V_i$. The group $\Sp(2g, \Z)$ acts on the set of all unordered symplectic splittings. Denote by $\mathcal{H}_g = \SL(2, \Z)^{\times g} \rtimes S_{g}$ the stabilizer of the unordered splitting $\V = \{V_1, \dots, V_n\}$ in $\Sp(2g, \Z)$, where $S_{g}$ is the symmetric group.
 
The group $\Stab_{\Mod(\S_g)} (N)$ preserves the splitting $\V$, therefore the image of the natural homomorphism $\Stab_{\Mod(\S_g)} (N) \to \Sp(2g, \Z)$ coincides with $\mathcal{H}_g$.
Consider the corresponding mapping $\eta: \Stab_{\Mod(\S_g)} (N) \twoheadrightarrow \mathcal{H}_g$. We check in Proposition \ref{prop} that $\ker (\eta) \subseteq \K_g$, so 
we have the commutative diagram 
\begin{equation}\label{maincd}
\scalebox{1}{
\xymatrix{
	1 \ar[r] & \K_g \ar[r] & \Mod(\S_g) \ar[r] & \mathcal{G}_g \ar@{->>}[rd] \ar[r] & 1\\
	&&&& \Sp(2g, \Z)  \\
	1 \ar[r] & \Stab_{\K_g} (N) \ar[r] \ar@{^{(}->}[uu] & \Stab_{\Mod(\S_g)} (N) \ar[r]^(0.7){\eta} \ar@{^{(}->}[uu] & \mathcal{H}_g \ar[r] \ar@{^{(}-->}[uu] \ar@{^{(}->}[ru] & 1
}
}
\end{equation}
Therefore we have the maps $\mathcal{H}_g \to \mathcal{G}_g$. Since the inclusion $\mathcal{H}_g \hookrightarrow \Sp(2g, \Z)$ is passing through $
\mathcal{G}_g$, then we have the inclusion $\mathcal{H}_g  \hookrightarrow \mathcal{G}_g$.
The second row of (\ref{maincd}) implies that there is the action of the group $\mathcal{H}_g = \SL(2, \Z)^{\times g} \rtimes S_{g}$ on $\P_g$. The action of $\SL(2, \Z)^{\times g}$ is trivial, therefore $\P_g$ is an $S_g$-module.
The first part of the main result is as follows.

\begin{theorem} \label{main1}
	There is an isomorphism of $\mathcal{G}_g$-modules $$\A_g \cong \Ind^{\mathcal{G}_g}_{\SL(2, \Z)^{\times g} \rtimes S_{g}} \P_{g}.$$
\end{theorem}

In order to describe the $S_g$-module $\P_{g}$ we need to introduce some notation. 
Denote by $\mathbf{T}_{g}$ the set of trees $\T$ with the following properties.\\
(1) $\T$ has $g$ leaves (vertices of degree one) marked by $1, \dots, g$. \\
(2) Degrees of all other vertices of $\T$ equal $3$.\\
We consider such trees up to an isomorphism preserving marking of the leaves. One can prove that $|\mathbf{T}_{g}| = 1 \cdot 3 \cdot 5 \cdot \dots \cdot (2g-5)$. For example, $\mathbf{T}_{3}$ consists of a single element. 

For each S-multicurve $M$ we consider the \textit{dual tree} $\T(M)$ i.e. the graph that has a vertex for each connected component of $\S_{g} \setminus M$ and two vertices are adjacent if and only if the corresponding connected components are adjacent to each other. Since each component of $M$ is separating it follows that $\T(M)$ is a tree. The tree $\T(M)$ has $g$ leaves corresponding to one-punctured tori; degrees of all other vertices equal $3$. By definition components of $M$ are ordered, so we also have the order on the set of curves that bound one-punctured tori on $\S_{g}$. Each leaf of $\T(M)$ corresponds to a component of $M$ with a number from $1$ to $g$, therefore the leaves of $\T(M)$ are numbered from $1$ to $g$. Hence, $\T(M)$ is an element of $\mathbf{T}_{g}$. For example, for the multicurve $\Delta$ in Fig. \ref{S-curve_1}, its dual tree $\T_{0} = \T(\Delta)$ is shown in Fig. \ref{Tree}.

\begin{figure}[h]
	\scalebox{1.5}{
		\begin{tikzpicture}
		\draw[] (0, 0) to (1, 0);
		\draw[] (1, 0) to (2, 0);
		\draw[][dotted] (2, 0) to (3, 0);
		\draw[] (3, 0) to (4, 0);
		\draw[] (4, 0) to (5, 0);
		
		\draw[] (1, 0) to (1, 1);
		\draw[] (2, 0) to (2, 1);
		\draw[] (3, 0) to (3, 1);
		\draw[] (4, 0) to (4, 1);
		
		\fill[black]  (0, 0) circle [radius=2pt];
		\fill[black]  (1, 0) circle [radius=2pt];
		\fill[black]  (2, 0) circle [radius=2pt];
		\fill[black]  (3, 0) circle [radius=2pt];
		\fill[black]  (4, 0) circle [radius=2pt];
		\fill[black]  (5, 0) circle [radius=2pt];
		\fill[black]  (1, 1) circle [radius=2pt];
		\fill[black]  (2, 1) circle [radius=2pt];
		\fill[black]  (3, 1) circle [radius=2pt];
		\fill[black]  (4, 1) circle [radius=2pt];
		
		\node[scale = 0.7][left] at (0, 0) {$1$};
		\node[scale = 0.7][above] at (1, 1) {$2$};
		\node[scale = 0.7][above] at (2, 1) {$3$};
		\node[scale = 0.7][above] at (3, 1) {$g-2$};
		\node[scale = 0.7][above] at (4, 1) {$g-1$};
		\node[scale = 0.7][right] at (5, 0) {$g$};
		
		\end{tikzpicture}}
	\caption{Dual tree $\T_{0} = \T(\Delta)$ to the S-multicurve $\Delta$.}
	\label{Tree}
\end{figure}
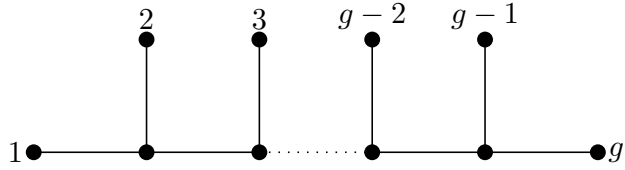

Recall that we have the fixed curves $\delta_{1}, \dots, \delta_{g}$ on $\S_{g}$ as in Fig. \ref{S-curve_1}.
For each $\T \in \mathbf{T}_{g}$ we can find a multicurve $\xi_{2} \cup \dots \cup \xi_{g-2}$ disjoint from $\delta_{1}, \dots, \delta_{g}$ and consisting of separating components such that $\T$ is the dual tree to the multicurve $\Delta_{\T} = \delta_{1}\cup \dots \cup \delta_{g} \cup \xi_{2} \cup \dots \cup \xi_{g-2}$. Such a multicurve $\xi_{2} \cup \dots \cup \xi_{g-2}$ is not unique, but we will prove that all such multicurves $\delta_{1}\cup \dots \cup \delta_{g} \cup \xi_{2} \cup \dots \cup \xi_{g-2}$ belong to the same $\K_g$-orbit, see Proposition \ref{correct}. Therefore the simple abelian cycle $\A_{\T} = \A(\Delta_{\T}) \in \H_{2g-3}(\K_g, \Z)$ is defined uniquely up to a sign. The sign of $\A_\T$ depends on the ordering of the curves $\xi_{2}, \dots, \xi_{g-2}$.

Let $h_{1} = 1, h_{2}, h_{3},  \dots \in \Sp(2g, \Z)$ be representatives of all left cosets $\Sp(2g, \Z) / \mathcal{H}_g$ and let $\hat{h}_1, \hat{h}_2, \hat{h}_3 \dots \in \Mod(\S_{g})$ be mapping classes that go to $h_1, h_2, h_3 \dots$ under the natural surjetive homomorphism $\Mod(\S_{g}) \twoheadrightarrow \Sp(2g, \Z)$. Gaifullin \cite[Theorem 1.3]{Gaifullin_J} proved that the abelian cycles 
$$\hat{h}_s\cdot \A_{\T_{0}}, \; \; \; s = 1, 2, 3, \dots$$
form a basis of a free $\Z[\wedge^{3} \H / \H]$-submodule of $\H_{2g-3}(\K_g, \Z)$. In particular, these simple abelian cycles are nonzero and generate a free abelian group.

\begin{definition}
	A triple of trees $\{\T_{1}, \T_{2}, \T_{3}\} \subseteq \mathbf{T}_{g}$ is called \textit{cyclic} if they differ only as shown in Fig. \ref{cyclic} (upper and lower vertices in Fig. \ref{cyclic} can be either leaves or not). 
\end{definition}

\begin{figure}[h]
	\scalebox{1.5}{
		\begin{tikzpicture}
		\coordinate (A1) at (0,0);
		\coordinate (B1) at (0,2);
		\coordinate (C1) at (1.5,2);
		\coordinate (D1) at (1.5,0);
		\coordinate (E1) at (0.5,1);
		\coordinate (F1) at (1,1);
		\fill[black]  (A1) circle [radius=2pt];
		\fill[black]  (B1) circle [radius=2pt];
		\fill[black]  (C1) circle [radius=2pt];
		\fill[black]  (D1) circle [radius=2pt];
		\fill[black]  (E1) circle [radius=2pt];
		\fill[black]  (F1) circle [radius=2pt];
		\draw{(A1) -- (E1)};
		\draw{(B1) -- (E1)};
		\draw{(C1) -- (F1)};
		\draw{(D1) -- (F1)};
		\draw{(E1) -- (F1)};
		\coordinate (K1) at (-0.25,-0.25);
		\coordinate (L1) at (0.25,-0.25);
		\coordinate (M1) at (-0.25,2.25);
		\coordinate (N1) at (0.25,2.25);
		\coordinate (O1) at (1.25,2.25);
		\coordinate (P1) at (1.75,2.25);
		\coordinate (R1) at (1.25,-0.25);
		\coordinate (S1) at (1.75,-0.25);
		\draw[dashed]{(A1) -- (K1)};
		\draw[dashed]{(A1) -- (L1)};
		\draw[dashed]{(B1) -- (M1)};
		\draw[dashed]{(B1) -- (N1)};
		\draw[dashed]{(C1) -- (O1)};
		\draw[dashed]{(C1) -- (P1)};
		\draw[dashed]{(D1) -- (R1)};
		\draw[dashed]{(D1) -- (S1)};
		
		\coordinate (A2) at (3,0);
		\coordinate (B2) at (3,2);
		\coordinate (C2) at (4.5,2);
		\coordinate (D2) at (4.5,0);
		\coordinate (E2) at (3.5,1);
		\coordinate (F2) at (4,1);
		\fill[black]  (A2) circle [radius=2pt];
		\fill[black]  (B2) circle [radius=2pt];
		\fill[black]  (C2) circle [radius=2pt];
		\fill[black]  (D2) circle [radius=2pt];
		\fill[black]  (E2) circle [radius=2pt];
		\fill[black]  (F2) circle [radius=2pt];
		\draw{(A2) -- (E2)};
		\draw{(B2) -- (F2)};
		\draw{(C2) -- (F2)};
		\draw{(D2) -- (E2)};
		\draw{(E2) -- (F2)};
		
		\coordinate (K2) at (2.75,-0.25);
		\coordinate (L2) at (3.25,-0.25);
		\coordinate (M2) at (2.75,2.25);
		\coordinate (N2) at (3.25,2.25);
		\coordinate (O2) at (4.25,2.25);
		\coordinate (P2) at (4.75,2.25);
		\coordinate (R2) at (4.25,-0.25);
		\coordinate (S2) at (4.75,-0.25);
		\draw[dashed]{(A2) -- (K2)};
		\draw[dashed]{(A2) -- (L2)};
		\draw[dashed]{(B2) -- (M2)};
		\draw[dashed]{(B2) -- (N2)};
		\draw[dashed]{(C2) -- (O2)};
		\draw[dashed]{(C2) -- (P2)};
		\draw[dashed]{(D2) -- (R2)};
		\draw[dashed]{(D2) -- (S2)};
		
		\coordinate (A3) at (6,0);
		\coordinate (B3) at (6,2);
		\coordinate (C3) at (7.5,2);
		\coordinate (D3) at (7.5,0);
		\coordinate (E3) at (6.5,1);
		\coordinate (F3) at (7,1);
		\fill[black]  (A3) circle [radius=2pt];
		\fill[black]  (B3) circle [radius=2pt];
		\fill[black]  (C3) circle [radius=2pt];
		\fill[black]  (D3) circle [radius=2pt];
		\fill[black]  (E3) circle [radius=2pt];
		\fill[black]  (F3) circle [radius=2pt];
		\draw{(A3) -- (E3)};
		\draw{(B3) -- (F3)};
		\draw{(C3) -- (E3)};
		\draw{(D3) -- (F3)};
		\draw{(E3) -- (F3)};
		
		\coordinate (K3) at (5.75,-0.25);
		\coordinate (L3) at (6.25,-0.25);
		\coordinate (M3) at (5.75,2.25);
		\coordinate (N3) at (6.25,2.25);
		\coordinate (O3) at (7.25,2.25);
		\coordinate (P3) at (7.75,2.25);
		\coordinate (R3) at (7.25,-0.25);
		\coordinate (S3) at (7.75,-0.25);
		\draw[dashed]{(A3) -- (K3)};
		\draw[dashed]{(A3) -- (L3)};
		\draw[dashed]{(B3) -- (M3)};
		\draw[dashed]{(B3) -- (N3)};
		\draw[dashed]{(C3) -- (O3)};
		\draw[dashed]{(C3) -- (P3)};
		\draw[dashed]{(D3) -- (R3)};
		\draw[dashed]{(D3) -- (S3)};

		\end{tikzpicture}}
	\caption{Cyclic triple of trees.}
	\label{cyclic}
\end{figure}
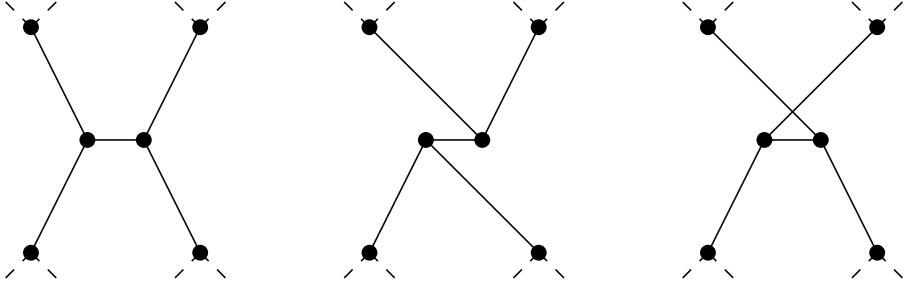

\begin{theorem} \label{main2}
	The abelian group $\P_{g}$ has a presentation where the generators are
	$\{\A_{\T} \; | \; \T \in \mathbf{T}_{g}\}$
	and the relations are
	\begin{equation}\label{relat}
	\{\A_{\T_1} + \A_{\T_2} + \A_{\T_3} = 0 \; | \; \{\T_{1}, \T_{2}, \T_{3}\} \mbox{ is a cyclic triple} \}.
	\end{equation}
\end{theorem}

\begin{remark}
	Recall that the signs of the simple abelian cycles $\A_T$ depend on the order of the components of the corresponding S-multicurve. 
	If $\{\T_{1}, \T_{2}, \T_{3}\}$ is a cyclic triple, then the corresponding S-multicurves $\Delta_{\T_1}$, $\Delta_{\T_2}$, $\Delta_{\T_3}$ differ by only one component. In formula (\ref{relat}) we mean that the components of these three S-multicurves are ordered such that the orderings coincide at $2g-4$ positions.
\end{remark}

\begin{remark}
	The ''hard part'' of  Theorem \ref{main2} is the fact that any relation between simple abelian cycles follows from the relations (\ref{relat}). However, the existence of such relations is not hard.
	For example, one can deduce the relation (\ref{relat}) from the Lantern relation \cite[Proposition 5.1]{Primer}. Our proof is based on Arnold's relations in the cohomology of the pure braid group.
\end{remark}

Also we find an explicit basis of $\P_{g}$. For each tree $\T \in \mathbf{T}_{g}$ we say that the leaf with number $g$ is the \textit{root}, so $\T$ is a \textit{rooted tree}. In this case for each vertex the set of its \textit{descendant leaves} is well defined.

\begin{definition} \label{balance}
	Let $\T \in \mathbf{T}_{g}$. A vertex of $\T$ of degree $3$ is called \textit{balanced} if the paths from it to the two descendant leaves with the two smallest numbers have no common edges. The tree $\T$ is called \textit{balanced} is all its vertices of degree $3$ are balanced. The set of all balanced trees is denoted by $\mathbf{T}^{b}_{g} \subseteq \mathbf{T}_{g}$.
\end{definition}

An example of a balanced tree in the case $g=7$ is shown in Fig. \ref{bTree}.

\begin{figure}[h]
	\scalebox{1.5}{
		\begin{tikzpicture}
		\draw[] (0, 4) to (0, 3);
		\draw[] (-2, 2) to (0, 3);
		\draw[] (2, 2) to (0, 3);
		\draw[] (-2, 2) to (-3, 1);
		\draw[] (2, 2) to (3, 1);
		\draw[] (-2, 2) to (-1, 1);
		\draw[] (2, 2) to (1, 1);
		\draw[] (-3, 1) to (-4, 0);
		\draw[] (-3, 1) to (-2, 0);
		\draw[] (3, 1) to (4, 0);
		\draw[] (3, 1) to (2, 0);
		
		\fill[black]  (0, 4) circle [radius=2pt];
		\fill[black]  (0, 3) circle [radius=2pt];
		\fill[black]  (2, 2) circle [radius=2pt];
		\fill[black]  (-2, 2) circle [radius=2pt];
		\fill[black]  (-3, 1) circle [radius=2pt];
		\fill[black]  (-1, 1) circle [radius=2pt];
		\fill[black]  (1, 1) circle [radius=2pt];
		\fill[black]  (3, 1) circle [radius=2pt];
		\fill[black]  (-4, 0) circle [radius=2pt];
		\fill[black]  (-2, 0) circle [radius=2pt];
		\fill[black]  (2, 0) circle [radius=2pt];
		\fill[black]  (4, 0) circle [radius=2pt];
		
		\node[scale = 0.7][left] at (-4, 0) {$3$};
		\node[scale = 0.7][right] at (-2, 0) {$5$};
		\node[scale = 0.7][right] at (-1, 1) {$2$};
		\node[scale = 0.7][left] at (1, 1) {$4$};
		\node[scale = 0.7][left] at (2, 0) {$6$};
		\node[scale = 0.7][right] at (4, 0) {$1$};
		\node[scale = 0.7][right] at (0, 4) {$7$};
		
		\end{tikzpicture}}
	\caption{An example of a balanced tree in the case $g=7$.}
	\label{bTree}
\end{figure}
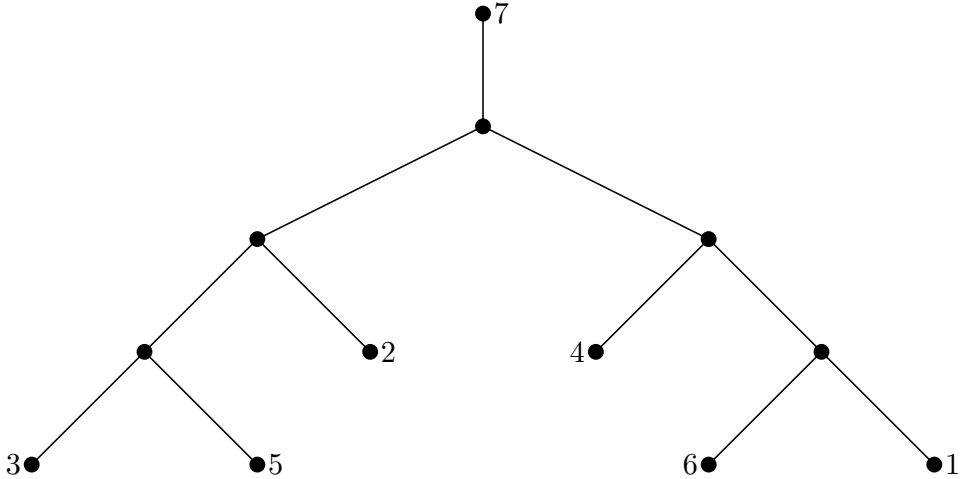

\begin{theorem} \label{main3}
	The simple abelian cycles
	$\{\A_{\T} \; | \;  \T \in \mathbf{T}^{b}_{g}\}$ form a basis of $\P_{g}$. We have $\rk \P_g = |\mathbf{T}^{b}_{g}| = (g-2)!$.
\end{theorem}

Theorems \ref{main1}, \ref{main2} and \ref{main3} provide a complete description of the group $\A_g$.

The author would like to thank his advisor Alexander A. Gaifullin for stating the problem, useful discussions and constant attention to this work. The author is a winner of the all-Russia mathematical August Moebius contest of graduate and undergraduate student papers and thanks the jury for the high praise of his work.

\section{Preliminaries and sketch of proof}

\subsection{Mapping class group of a surface with punctures and boundary components} Let $\S$ be an oriented surface, possibly with punctures and boundary components. We do not assume that $\S$ is connected. However, we require $\H_*(\S, \Q)$ be a finite dimensional vector space. The mapping class group of $\S$ is defined as $\Mod(\S) = \pi_{0}(\Homeo^{+}(\S, \partial\S))$, where $\Homeo^{+}(\S, \partial \S)$ is the group of orientation-preserving homeomorphisms of $\S$ that restrict to the identity on $\partial \S$. 
By $\PMod(\S) \subseteq \Mod(\S)$ we denote the \textit{pure mapping class group} of $\S$, i.e. the subgroup consisting of those elements fixing each of the punctures and each of the connected components. We have the exact sequence
\begin{equation} \label{Pmod}
1 \rightarrow \PMod(\S_{g, n}^b) \rightarrow \Mod(\S_{g, n}^b) \rightarrow S_n \rightarrow 1,
\end{equation}
where by $\S_{g, n}^b$ we denote the connected genus $g$ surface with $n$ punctures and $b$ boundary components.
For example, the pure mapping class group of the disk with $n$ punctures is precisely the \textit{pure braid group} $\PB_n = \PMod(\S_{0, n}^1)$.

\subsection{The Birman-Lubotzky–McCarthy exact sequence}

Let $M$ be a multicurve on $\S_g$. Then there is the following Birman-Lubotzky–McCarthy exact sequence (see \cite[Lemma 2.1]{Lubotzky})
\begin{equation} \label{LB}
1 \rightarrow G(M) \rightarrow \Stab_{\Mod(\S_{g})}(M) \rightarrow \Mod(\S_{g} \setminus M) \rightarrow 1,
\end{equation}
where $G(M)$ is the group generated by Dehn twists about the components of $M$.

Take $N = \delta_1 \cup \dots \cup \delta_g$ as in Fig. \ref{S-curve_1} and consider the group $\Stab_{\Mod(\S_{g})}(N)$. We have $\S_{g} \setminus N = \S_{0, g} \sqcup X_1 \sqcup \dots \sqcup X_g$, where $X_i$  is the one-punctured torus bounded by $\delta_i$. Since $\Mod(X_i) \cong \SL(2, \Z)$,  (\ref{Pmod}) and (\ref{LB}) imply the existence of the following commutative diagram.

\begin{equation} \label{LB0}
\scalebox{1}{
	\xymatrix{	
		&\ker \eta \ar@{-->>}[rr] \ar@{^{(}->}[rd] &&\PMod(\S_{0, g}) \ar@{^{(}->}[d] &\\
		1 \ar[r] & G(N) \ar@{^{(}-->}[u] \ar[r] & \Stab_{\Mod(\S_{g})}(N) \ar[r] \ar@{->>}[rd]_{\eta} &  \SL(2, \Z)^{\times g} \rtimes \Mod(\S_{0, g}) \ar[r] \ar@{->>}[d]  & 1\\
		&&& \SL(2, \Z)^{\times g} \rtimes S_{g}&
}}
\end{equation}
The diagram (\ref{LB0}) yields the exact sequence
\begin{equation} \label{LB3}
1 \rightarrow G(N) \rightarrow \ker \eta \rightarrow \PMod(\S_{0, g}) \rightarrow 1.
\end{equation}

\begin{prop} \label{prop}
	The following sequence is exact.
	\begin{equation}\label{exact}
	\scalebox{1}{
		\xymatrix{
			1 \ar[r] & \Stab_{\K_g} (N) \ar[r]  & \Stab_{\Mod(\S_g)} (N) \ar[r]^{\eta}  &  \SL(2, \Z)^{\times g} \rtimes S_{g}  \ar[r]  & 1.
		}
	}
	\end{equation}
\end{prop}
\begin{proof}
	First let us show that $\Stab_{\K_g} (N) \subseteq \ker \eta$. Indeed, any element $\phi \in \Stab_{\K_g} (N)$ stabilises each component of $N$, so it also stabilises each $X_i$. Since $\K_g$ is contained in the Torelli group, it follows that the restriction of $\phi$ to $\Mod(X_i) \cong \SL(2, \Z) \subset \Sp(2g, \Z)$ is trivial for all $i$. 
	
	Let us prove the opposite inclusion.
    The groups $G(M)$ and $\PMod(\S_{g,n})$ are generated by Dehn twists about separating curves. The exact sequence (\ref{LB3}) implies that the same it true for $\ker \eta$, therefore $\ker \eta \subseteq \K_g$. Hence $\ker \eta \subseteq \Stab_{\K_g} (N)$. 
\end{proof}

\begin{lemma}\label{lemma22}
	There is an isomorphism
	\begin{equation} \label{Stab}
	\Stab_{\K_g}(N) \cong \Z^{g-1} \times \PB_{g-1}.
	\end{equation}
\end{lemma}
\begin{proof}
	We need the following fact.
	\begin{fact} \cite[Section 9.3]{Primer} \label{factPB}
	The center of the group $\PB_{g-1}$ is the infinite cyclic group, which is generated by the Dehn twist about the boundary curve. Moreover, we have the split exact sequence
	\begin{equation*}
	\scalebox{1}{
		\xymatrix{
			1 \ar[r] & \Z \ar[r]^(0.4){j_1} & \PB_{g-1} \ar[r] &  \PMod(\S_{0, g}) \ar[r] & 1,
	}}
	\end{equation*}
	where $j_1$ is the inclusion of the center of $\PB_{g-1}$.
	\end{fact}
	Consider the obvious map $$j: \Z^g \cong \Z^{g-1} \times \Z \hookrightarrow \Z^{g-1} \times \PB_{g-1},$$
	where the restriction of $j$ on the first factor is the identity isomorphism and the restriction of $j$ on the second factor is $j_1$.
	Fact \ref{factPB} and the exactness of (\ref{LB3}) implies that in order to finish the proof of Lemma \ref{lemma22} we need to construct the map $\psi: \Z^{g-1} \times \PB_{g-1} \to \Stab_{\K_g} (N)$ such that the following diagram commutes.
	\begin{equation*}\label{5lemma}
	\scalebox{1}{
		\xymatrix{
			1 \ar[r] & \Z^g \ar[r] & \Z^{g-1} \times \PB_{g-1} \ar[r] \ar@{-->}[d]^{\psi} &  \PMod(\S_{0, g}) \ar[r] & 1\\
			1 \ar[r] & G(N) \ar[r] \ar@{=}[u] & \Stab_{\K_g} (N) \ar[r]  & \PMod(\S_{0, g})  \ar[r] \ar@{=}[u]  & 1.
		}
	}
	\end{equation*}

\begin{figure}[h]
	\scalebox{1.8}{
		\begin{tikzpicture}
        \draw[] (2, 0) arc (0:-180:2);
        \draw[] (0, -1.5) circle (0.2);
		
		\draw[][red] (-2, 0) to [out=-40, in=-140] (2, 0);
		\draw[][red] (-2, 0) to [out=40, in=140] (2, 0);
		\node[red][scale=0.7] at (-0.8, -0.82) {$\delta_{g}$};

		\draw[][red] (-1.2, 0) to [out=-40, in=-140] (-0.7, 0);
		\draw[dashed][red] (-1.2, 0) to [out=40, in=140] (-0.7, 0);
		\draw[] (-1.2, 0) to (-1.2, 0.2);
		\draw[] (-0.7, 0) to (-0.7, 0.2);
		\draw[] (-1.2, 0.2) arc (180:0:0.25);
		\draw[] (-0.95, 0.3) circle (0.1);
		\node[red][scale=0.7] at (-0.95, -0.28) {$\delta_{1}$};
		
		\draw[][red] (-0.6, 0) to [out=-40, in=-140] (-0.1, 0);
		\draw[dashed][red] (-0.6, 0) to [out=40, in=140] (-0.1, 0);
		\draw[] (-0.6, 0) to (-0.6, 0.2);
		\draw[] (-0.1, 0) to (-0.1, 0.2);
		\draw[] (-0.6, 0.2) arc (180:0:0.25);
		\draw[] (-0.35, 0.3) circle (0.1);
		\node[red][scale=0.7] at (-0.35, -0.28) {$\delta_{2}$};
		
		\draw[dotted][thick][red] (0.1, 0) to (0.6, 0);
		
		\draw[][red] (0.7, 0) to [out=-40, in=-140] (1.2, 0);
		\draw[dashed][red] (0.7, 0) to [out=40, in=140] (1.2, 0);
		\draw[] (0.7, 0) to (0.7, 0.2);
		\draw[] (1.2, 0) to (1.2, 0.2);
		\draw[] (0.7, 0.2) arc (180:0:0.25);
		\draw[] (0.95, 0.3) circle (0.1);
		\node[red][scale=0.7] at (0.95, -0.28) {$\delta_{g-1}$};

		\end{tikzpicture}}
	\caption{}
	\label{iso}
\end{figure}

We define $\psi$ as follows. The generator of the $i^{\rm th}$ factor of $\Z^{g-1}$ maps to $T_{\delta_i}$. In order to define the restriction of $\psi$ on the factor $\PB_{g-1}$ let us identify $\S_g$ with the surface shown in Fig \ref{iso}. We have the disk bounded by $\delta_g$ with $g-1$ handles bounded by $\delta_1, \dots, \delta_{g-1}$. We can replace  all these handles by punctures and identify the group $\PB_{g-1}$ with the corresponding group $\PMod(\S^1_{0, g-1})$. Then we extend the mapping classes in $\PMod(\S^1_{0, g-1})$ to the handles so that the handles do not rotate.

Since the pure braid group is generated by Dehn twists about separating curves it follows that the image of $\psi$ is contained in $\K_g$. The 5-lemma completes the proof of Lemma \ref{lemma22}.
\end{proof}

\subsection{Simple abelian cycles}

Recall that for an S-multicurve $M = \gamma_{1} \cup \dots \cup \gamma_{2g-3}$ on $\S_g$ there is the corresponding simple abelian cycle $\A(M) = \A(T_{\gamma_{1}}, \dots, T_{\gamma_{2g-3}}) \in \H_{2g-3}(\K_g, \Z)$. We have already constructed the simple abelian cycles $\A_{\T} = \A(\Delta_\T)$ for $\T \in \mathbf{T}_{g}$.
\begin{prop} \label{correct}
	Let $\Delta_\T = \delta_{1} \cup \dots \cup \delta_{g} \cup \xi_{2} \cup \dots \cup \xi_{g-2}$ and $\Delta'_\T = \delta_{1} \cup \dots \cup \delta_{g} \cup \xi'_{2} \cup \dots \cup \xi'_{g-2}$ be two S-multicurves with the same dual tree $\T$. Then $\Delta_{\T}$ and $\Delta'_\T$ belong to the same $\K_g$-orbit (up to a permutation of the components).
    In particular, the simple abelian cycles $\{\A_{\T} \; | \; \T \in \mathbf{T}_g\}$ are well defined.
\end{prop}
\begin{proof}
	Since $\Delta_\T$ and $\Delta'_\T$ have the same dual tree, there is an element $\phi \in \Mod(\S_{g})$ such that after a permutation of $\xi_2, \dots, \xi_{g-2}$ we have $\phi(\xi_i) = \xi'_i$ and $\phi(\delta_j) = \delta_j$ for all $i$ and $j$. Also, we can assume that $\phi|_{X_j} = \id$ for all $1 \leq j \leq g$. Then $\phi \in \ker \eta$ (see the exact sequence (\ref{exact})), so $\phi \in \K_g$. Therefore all such S-multicurves $\Delta_{\T}$ belong to the same $\K_g$-orbit, so the simple abelian cycle $\A_{\T} = \A(\Delta_{\T}) \in \H_{2g-3}(\K(\S_{g}, \Z))$ is defined uniquely up to a sign.
\end{proof}

\begin{prop} \label{AC}
	 The simple abelian cycles $\{\A_{\T} \; | \; \T \in \mathbf{T}_g\}$ generate $\A_g$ as a $\Z[\mathcal{G}_g]$-module.
\end{prop}
\begin{proof}
	Consider an S-multicurve $M' = \gamma_{1} \cup \dots \cup \gamma_{2g-3}$. We can assume that the curves $\gamma_1, \dots, \gamma_g$ bound one-punctured tori on $\S_g$. There is an element $\phi \in \Mod(\S_{g})$ such that $\phi(\gamma_j) = \delta_j$ for all $j$. Then $\phi \cdot \A(M') = \pm \A_{\T(M')}$, so $\A(M') = \pm \phi^{-1} \cdot \A_{\T(M')}$. This implies the proposition.
\end{proof}

\subsection{Sketches of the proofs of Theorems \ref{main1}, \ref{main2} and \ref{main3}}

By Lemma \ref{lemma22} we can consider the simple abelian cycles 
$\{\A_{\T} \; | \; \T \in \mathbf{T}_g\}$ as an elements of the group $\H_{2g-3}(\Stab_{\K_g}(N), \Z) \cong \H_{2g-3}(\Z^{g-1} \times \PB_{g-1})$.
The following proposition will be proved in Section \ref{sec3}. Its proof is based on Arnold's relations in the cohomology of the pure braid group.
\begin{prop} \label{th1}
	The abelian group $\H_{2g-3}(\PB_{g-1} \times \Z^{g-1}, \Z)$ is generated by the elements $\A_\T$, where $\T \in \mathbf{T}_{g}$. All relations among this generators follows from the relations 
	$$\A_{\T_1} + \A_{\T_2} + \A_{\T_3} = 0$$
	for each cyclic triple $\{\T_{1}, \T_{2}, \T_{3}\}$.
\end{prop}
The next result will be proved in Sections 4 and 5. The proof is based on the spectral sequence for the action of $\K_g$ on the contractible complex of cycles, introduced by Bestvina, Bux and Margalit in \cite{Bestvina}, and certain new complexes which will be constructed below.
\begin{prop} \label{th2}
	Let $f_{1} = 1, f_{2}, f_{3},  \dots \in \mathcal{G}_g$ be representatives of all left cosets $\mathcal{G}_g / \mathcal{H}_g$ and let $\hat{f}_1, \hat{f}_2, \hat{f}_3 \dots \in \Mod(\S_{g})$ be their lifts in $\Mod(\S_{g})$. Then the inclusions
	\begin{equation*}
	i_s: \Stab_{\K_g}(\hat{f}_s \cdot N) \hookrightarrow \K_g, \;\;\; s \in \mathbb{N}
	\end{equation*}
	induce an injective homomorphism
	\begin{equation} \label{eq0}
    \bigoplus_{s \in \mathbb{N}} \H_{2g-3}(\Stab_{\K_g}(\hat{f}_s \cdot N), \Z) \hookrightarrow \H_{2g-3}(\K_g, \Z).
	\end{equation}
\end{prop}
\begin{proof}[Proof of Theorem \ref{main2}]
By Proposition \ref{th2} the map $$i_1: \H_{2g-3}(\Stab_{\K_g}(N), \Z) \hookrightarrow \H_{2g-3}(\K_g, \Z)$$ is injective. Since we have $\P_{g} = i_1(\H_{2g-3}(\Stab_{\K_g}(N), \Z))$, Proposition \ref{th1} implies the required result.	
\end{proof}

\begin{proof}[Proof of Theorem \ref{main1}]
By Proposition \ref{th2} and Proposition \ref{AC} we obtain that (\ref{eq0}) induces an isomorphism
\begin{equation*}
\bigoplus_{s \in \mathbb{N}} \H_{2g-3}(\Stab_{\K_g}(\hat{f_{s}} \cdot N), \Z) \cong \A_g,
\end{equation*}
so
\begin{equation*}
\A_g = \bigoplus_{s \in \mathbb{N}} \hat{f}_{s} \cdot \P_{g}.
\end{equation*}
Therefore, by definition of induced module we have
\begin{equation*}
\hspace{1.1cm} \A_g \cong \P_{g} \otimes_{\mathcal{H}_g} \Z[\mathcal{G}_g] = \Ind^{\mathcal{G}_g}_{\mathcal{H}_g} \P_{g}.
\end{equation*}
\end{proof}
Theorem \ref{main3} will be deduced from Proposition \ref{th1} in Section \ref{sec3}.

\section{Proof of Proposition \ref{th1}} \label{sec3}

\subsection{Cohomology of the pure braid group} In order to prove Proposition \ref{th1} we conveniently consider the pure braid group on $g-1$ strands $q_1, \dots, q_{g-1}$. Let us recall Arnold's results on the structure of the ring $\H^{*}(\PB_{g-1}, \Z)$. The group $\PB_{g-1}$ has standard set of generators $a_{i, j}$ for $1 \leq i < j \leq g-1$. These elements are the Dehn twists about curves enclosing $i^{th}$ and $j^{\rm th}$ strands (see Fig. \ref{braid}). Denote by $h_{i, j} \in \H_{1}(\PB_{g-1}, \Z)$ the corresponding homology classes.  We denote by $\{w_{i, j}\}$ the dual basis of $\H^1(\PB_{g-1}, \Z)$. These cohomology classes can be interpreted as the following homomorphisms
\begin{equation} \label{braidcoh}
w_{i, j}: \PB_{g-1} \to \PB_2 \cong \Z
\end{equation}
given by forgetting all strands besides $q_i$ and $q_j$. It is convenient to put $w_{j, i} = w_{i, j}$. 

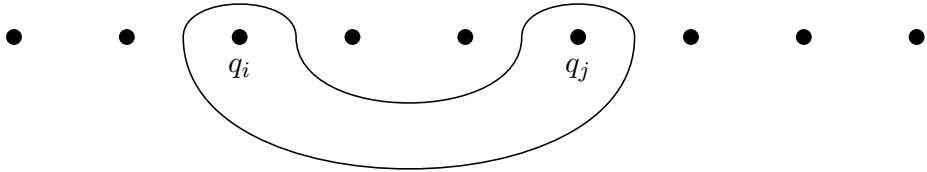
\begin{figure}[h]
	\scalebox{1.5}{
		\begin{tikzpicture}
		
		\draw[] (-2.5, 0) to [out=90, in=90] (-1.5, 0);
		\draw[] (0.5, 0) to [out=90, in=90] (1.5, 0);
		\draw[] (-1.5, 0) to [out=-90, in=-90] (0.5, 0);
		\draw[] (-2.5, 0) to [out=-90, in=-90] (1.5, 0);

		\fill[black]  (-4, 0) circle [radius=2pt];
		\fill[black]  (-3, 0) circle [radius=2pt];
		\fill[black]  (-2, 0) circle [radius=2pt];
		\fill[black]  (-1, 0) circle [radius=2pt];
		\fill[black]  (0, 0) circle [radius=2pt];
		\fill[black]  (1, 0) circle [radius=2pt];
		\fill[black]  (2, 0) circle [radius=2pt];
		\fill[black]  (3, 0) circle [radius=2pt];
		\fill[black]  (4, 0) circle [radius=2pt];

		\node[scale = 0.7][below] at (-2, -0.1) {$q_i$};
		\node[scale = 0.7][below] at (1, -0.1) {$q_j$};

		\end{tikzpicture}}
	\caption{The element $a_{i, j}$ is the Dehn twist about the shown curve.}
	\label{braid}
\end{figure}

\begin{theorem} \cite[Theorem 1]{Arnold} \label{Arnold1}
	The ring $\H^{*}(\PB_{g-1}, \Z)$ is the exterior graded algebra with ${g-1}\choose{2}$ generators $w_{i, j}$ of degree $1$, satisfying ${g-1}\choose{3}$ relations
	\begin{equation*} \label{presentanion}
	w_{k, l}w_{l, m} + w_{l, m}w_{m, k} + w_{m, k}w_{k, l} = 0
	\end{equation*} 
	for all $1 \leq k < l < m \leq n$.
\end{theorem}
\begin{corollary} \cite[Corollary 3]{Arnold} \label{corAr}
	The products
	\begin{equation} \label{basis}
	w_{k_{1}, l_{1}} w_{k_{2}, l_{2}} \cdots w_{k_{p}, l_{p}} \mbox{, where } k_{i}< l_{i} \mbox{ and } l_{1}<\cdots<l_{p},
	\end{equation} 
	form an additive basis of $\H^{*}(\PB_{g-1}, \Z)$.
\end{corollary}

Corollary \ref{corAr} implies that the products
\begin{equation} \label{basistop}
w_{k_{1}, 2} w_{k_{2}, 3} \cdots w_{k_{g-2}, g-1} \mbox{ where } k_{i} \leq i
\end{equation}
form an additive basis of $\H^{g-2}(\PB_{g-1}, \Z)$.
We denote the cohomology class (\ref{basistop}) by $W_{k} = W_{k_{1}, \dots, k_{g-2}}$, where $k = (k_{1}, \dots, k_{g-2})$. We denote by $\mathbf{K}_g$ the set of all sequences $k = (k_{1}, \dots, k_{g-2})$ satisfying $1 \leq k_{i} \leq i$.

\subsection{Abelian cycles in $\H_{g-2}(\PB_{g-1}, \Z)$} Corollary \ref{corAr} implies that $\cd(\PB_{n}, \Z) \geq n-1$. In fact we have $\cd(\PB_{n}) = n-1$. Indeed, let $M_n$ be the the ordered configuration space of $n$ points on the disk. This space is aspherical, $M_n \simeq K(\PB_{n}, 1)$. We have the fiber bundle $M_n \to M_{n-1}$ where the fiber is homotopy equivalent to the wedge of $n-1$ circles. Hence by induction we obtain that $M_n$ is homotopy equivalent to a $(n-1)$-dimensional CW-complex. Therefore we have $\cd(\PB_{n}, \Z) \leq n-1$.

Therefore the isomorphism (\ref{Stab}) implies
\begin{equation} \label{homology}
\H_{2g-3}(\Stab_{\K_g}(N), \Z) \cong \H_{2g-3}(\Z^{g-1} \times \PB_{g-1}, \Z) \cong \H_{g-2}(\PB_{g-1}, \Z).
\end{equation}
Let us recall the construction of the isomorphism $\Stab_{\K_g}(N) \cong \Z^{g-1} \times \PB_{g-1}$.
We consider the surface $\S_{0, g-1}^1$ is given by replacing the boundary components corresponding to the curves $\delta_{1}, \dots, \delta_{g-1}$ on $\S_{0, g} \subset \S_{g}$ by the punctures $q_1, \dots, q_{g-1}$. Hence we obtain the pure braid group $\PB_{g-1} = \PMod(\S_{0, g-1}^1)$. The $i^{\rm th}$ factor in $\Z^{g-1}$ is generated by $T_{\delta_i}$.

Consider a simple abelian cycle $$\A_{\T} = \A(T_{\delta_{1}}, \dots, T_{\delta_{g}}, T_{\xi_{2}}, \dots T_{\xi_{g-2}}) \in \H_{2g-3}(\Stab_{\K_g}(N), \Z)$$
for some $\T \in \mathbf{T}_{g}$.
Isomorphism (\ref{homology}) sends $\A_{\T}$ to the abelian cycle
\begin{equation*} \label{hatcyc}
\A(T_{\delta_{g}}, T_{\xi_{2}}, \dots T_{\xi_{g-2}}) \in \H_{g-2}(\PB_{g-1}, \Z),
\end{equation*}
Let us denote $\xi_1 = \delta_g$ and
$$\hat{\A}_\T = \A(T_{\xi_{1}}, T_{\xi_{2}}, \dots T_{\xi_{g-2}}) = \A(T_{\delta_{g}}, T_{\xi_{2}}, \dots T_{\xi_{g-2}}) \in \H_{g-2}(\PB_{g-1}, \Z).$$

Any simple closed curve on $\S_{0, g-1}^1$ divides it into two parts. 
We say that a puncture $q$ is \textit{enclosed} by a curve $\gamma$ on $\S_{0, g-1}^1$ if $q$ is contained in the part which does not contain the boundary component.
For $k \in \mathbf{K}_g$ define the matrix $X_{k, \T} \in \Mat_{(g-2) \times (g-2)}(\Z)$ by
\begin{equation}\label{matrix}
(X_{k, \T})_{i, j}=
\begin{cases}
1 \mbox{ if the punctures }  q_{k_i} \mbox{ and } q_{i+1} \mbox{ are enclosed by }  \xi_{j},  \\
0, \mbox{ otherwise.}
\end{cases}
\end{equation}

\begin{lemma} \label{pair}
	Let $k \in \mathbf{K}_g$ and $\T \in \mathbf{T}_g$. Then $\langle W_{k}, \hat{\A}_\T \rangle = (-1)^{{g-2}\choose{2}} \det(X_{k, \T})$.
\end{lemma}
\begin{proof}
	Consider a free abelian group $\Z^{g-2} = \langle c_{1}, \dots, c_{g-2}\rangle$ and the homomorphism $f: \Z^{g-2} \rightarrow \PB_{g-1}$ given by $c_{i} \mapsto T_{\xi_{i}}$. Denote by $\mu_{g-2}$ the standard generator of the group $H_{g-2}(\Z^{g-2}, \Z)$. We have
		\begin{equation*}
	\langle W_{k}, \hat{\A}_\T \rangle = \langle W_{k}, f_{*}(\mu_{g-2}) \rangle = \langle f^{*}W_{k}, \mu_{g-2} \rangle = \langle (f^{*}w_{k_{1}, 2}) \cdots (f^{*}w_{k_{g-2}, g-1}), \mu_{g-2} \rangle =
	\end{equation*}
	\begin{equation*}
	= (-1)^{{g-2}\choose{2}} \det(\langle f^{*}w_{k_{i}, i+1}, c_{j} \rangle)_{i, j = 1}^{g-2} = (-1)^{{g-2}\choose{2}} \det(\langle w_{k_{i}, i+1}, f_{*}c_{j} \rangle)_{i, j = 1}^{g-2} =
	\end{equation*}
	\begin{equation*}
	= (-1)^{{g-2}\choose{2}} \det(\langle w_{k_{i}, i+1}, T_{\xi_{j}} \rangle)_{i, j = 1}^{g-2} = (-1)^{{g-2}\choose{2}} \det(X_{k, \T}).
	\end{equation*}
	The last equality comes from the following corollary of formula (\ref{braidcoh}):
	\begin{equation*}
	\langle w_{k, l} , T_{\xi_{j}} \rangle = 
	\begin{cases}
	1 \mbox{ if the punctures } q_k \mbox{ and } q_l \mbox{ are enclosed by }  \xi_{j},  \\
	0, \mbox{ otherwise.}
	\end{cases}
	\end{equation*}
\end{proof}
Let us denote by $\{D_k \in \H_{g-2}(\PB_{g-1}, \Z) \; | \; k \in \mathbf{K}_g\}$ the dual basis to $\{W_k \; | \; k \in \mathbf{K}_g\}$.
\begin{corollary} \label{cor3}
	 Let $\T \in \mathbf{T}_g$. Then $\hat{\A}_\T = \sum_{k \in \mathbf{K}_g} (-1)^{{g-2}\choose{2}} \det(X_{k, \T}) D_{k}$.
\end{corollary}

\subsection{Balanced trees}

Recall that we consider the elements of $\mathbf{T}_{g}$ as marked trees such that the leaf with number $g$ is the root. Also we have already defined the subset $\mathbf{T}_g^b \subseteq \mathbf{T}_{g}$ of balanced trees. Take any $k \in \mathbf{K}_g$. Our goal is to construct a balanced tree $\T_k \in \mathbf{T}_{g}$ such that $\hat{\A}_{\T_k} = (-1)^{{g-2}\choose{2}} D_k$ and such that the map $k \mapsto \T_k$ is a bijection between the sets $\mathbf{K}_g$ and $\mathbf{T}_{g}^b$. First let us construct the map $k \mapsto \T_k$ (then we will check that $\hat{\A}_{\T_k} = (-1)^{{g-2}\choose{2}} D_k$, see Theorem \ref{cyctree}).

\begin{construction} \label{constr}
We construct curves $\xi_{1}, \dots, \xi_{g-2}$ such that $\hat{\A}_{\T_k} = \A(T_{\xi_{1}}, T_{\xi_{2}}, \dots T_{\xi_{g-2}})$ by induction on $g$. The case $g=3$ trivial since we have $|\mathbf{T}_3^b|=|\mathbf{T}_3|=|\mathbf{K}_3|=1$. Let us prove the induction step from $g-1$ to $g$. Consider any $k = (k_1, \dots, k_{g-2}) \in \mathbf{K}_g$ with $g > 3$. Let $\xi_{g-2}$ be a curve enclosing exactly two points $q_{k_{g-2}}$ and $q_{g-1}$. Let us remove the curve $\xi_{g-2}$ with its interior and denote the corresponding puncture by $q'_{k_{g-2}}$. Also take $q'_i = q_i$ for $i \leq g-2$ and $i \neq k_{g-2}$. We obtain a disk with $g-2$ punctures $q'_1, \dots, q'_{g-2}$ and $k' = (k_1, \dots, k_{g-3}) \in \mathbf{K}_{g-1}$. The induction hypothesis implies that there is a balanced tree $\T_{k'} \in \mathbf{T}_{g-1}^b$ corresponding to $k'$ given by some curves $\xi_1, \dots, \xi_{g-3}$. Now consider the curves $\xi_1, \dots, \xi_{g-3}, \xi_{g-2}$ and denote the dual tree by $\T_k \in \mathbf{T}_g$. It remains to show that $\T_k$ is balanced. Indeed, since the vertex $q_{g-1}$ has the greatest number, all common vertices of $\T_{k'}$ and $\T_k$ are  balanced. Also, this property holds for the vertex of $\T_k$ corresponding to the curve $\xi_{g-2}$, because it has only two descendant leaves. This implies the induction step.	
\end{construction}

Since for different $k, k' \in \mathbf{K}_g$ the correspondent trees $\T_k$ and $T_{k'}$ are also different, it follows that the map $k \mapsto \T_k$ given by Construction \ref{constr} is injective. Moreover, direct computation shows that $|\mathbf{K}_g| = (g-2)!$.
Therefore in order to prove that this map is a surjection to $\mathbf{T}_{g}^b$ it suffices to show that $|\mathbf{T}_g^b| = (g-2)!$. We use the induction on $g$; the base case $g=3$ is trivial. Consider a balanced tree $\T \in \mathbf{T}_{g}^b$ with $g \geq 4$. Let $q_1, \dots, q_{g-1}$ be its leaves (besides the root). Let $p$ be the vertex adjacent to $q_{g-1}$. Since $\T$ is balanced, another descendant vertex of $p$ is a leaf $q_i$ for some $1 \leq i \leq g-2$. Let us remove the vertices $q_i$ and $q_{g-1}$ (with the incident edges) and set $q_i=p$; denote the obtained tree by $\T'$. Then $\T' \in \mathbf{T}_{g-1}^b$. Since $|\mathbf{T}_{g-1}^b| = (g-3)!$ and there is $(g-2)$ possibilities to choose $i$ we have $|\mathbf{T}_{g}^b| = (g-2)!$. This implies the induction step.

\begin{theorem} \label{cyctree}
	Suppose that $k \in \mathbf{K}_g$. Then $\hat{\A}_{\T_k} = (-1)^{{g-2}\choose{2}} D_k$.
\end{theorem}
\begin{proof}
	By Corollary \ref{cor3} it suffices to show that for any $k' \in \mathbf{K}_g$ we have $\det(X_{k', \T_k}) = 1$ if $k = k'$ and $\det(X_{k', \T_k}) = 0$ otherwise.
\begin{lemma}\label{l1}
	Let $k \in \mathbf{K}_g$. Then $\det(X_{k, \T_k}) = 1$.
\end{lemma}
\begin{proof}
     Let $\xi_{1}, \dots, \xi_{g-2}$ be a multicurve with dual tree $\T_k$. By Construction \ref{constr} we have that the punctures $q_{k_i}$ and $q_{i+1}$ are enclosed by the curve $\xi_i$ for all $1 \leq i \leq g-2$. Indeed, for $i = g-2$ this follows from the construction of the curve $\xi_{g-2}$, and for $i < g-2$ this follows by the induction on $g$. Therefore $(X_{k, \T_k})_{i, i} = 1$ for all $i$.
    
    Now let us check that $(X_{k, \T_k})_{i, j} = 0$ whenever $i < j$. Indeed, for $j = g-2$ this follows from the construction of the curve $\xi_{g-2}$, and for $j < g-2$ this follows by the induction on $g$.
    Therefore $X_{k, \T_k}$ is lower unitriangular, so $\det(X_{k, \T_k}) = 1$.
\end{proof}
\begin{lemma}\label{l2}
	Let $k, k' \in \mathbf{K}_g$ and $k \neq k'$. Then $\det(X_{k', \T_k}) = 0$.
\end{lemma}
\begin{proof}
	Define $s = \max\{i \; | \; k_{i} \neq k'_{i}\}$.
	Let us check that the matrix $X_{k', \T_k}$ has the form (\ref{matr}) where the $s^{\rm th}$ column is highlighted. This would immediately imply $\det(X_{k', \T_k}) = 0$.
	\begin{equation} \label{matr}
	X_{k', \T_k} = \begin{pmatrix}
	*&*&\cdots&*&\vrule \: 0 \: \vrule&0&0&\cdots&0\\
	*&*&\cdots&*&\vrule \: 0 \: \vrule&0&0&\cdots&0\\
	\colon&\colon&\ddots& \colon &\vrule \, \, : \, \, \vrule&\colon&\colon&\ddots&\colon\\
	*&*&\cdots&*&\vrule \: 0 \: \vrule&0&0&\cdots&0\\
	*&*&\cdots&*&\vrule \: 0 \: \vrule&0&0&\cdots&0\\
	*&*&\cdots&*&\vrule \: * \: \vrule&1&0&\cdots&0\\
	*&*&\cdots&*&\vrule \: * \: \vrule&*&1&\cdots&0\\
	\colon&\colon&\ddots&\colon&\vrule \, \, : \, \, \vrule&\colon&\colon&\ddots&\colon\\
	*&*&\cdots&*&\vrule \: * \: \vrule&*&*&\cdots&1\\
	\end{pmatrix}.
	\end{equation}
	
	We have $k_i = k'_i$ for all $i > s$, so the similar arguments as in the proof of Lemma \ref{l1} show that the last $g-2-s$ columns have the required form. 
	
	Let $\T_k$ be given by curves $\xi_{1}, \dots, \xi_{g-2}$ as in Construction \ref{constr}. 
    Recall the construction of the curve $\xi_{s}$. At this step we have the punctures $q'_1, \dots, q'_{s+1}$. Some of them coincide with $q_i$, others are identified with the interiors of some curves $\xi_i$, with $i \geq s+1$. Nevertheless, $q_i$ is enclosed by $\xi_s$ if and only if  $q'_i$ is enclosed by $\xi_s$ for all $1 \leq i \leq s+1$.
	
	By Construction \ref{constr} $\xi_s$ is a curve enclosing exactly two punctures $q'_{s+1}$ and $q'_{k_{s}}$ from the set $\{q'_1, \dots, q'_{s+1}\}$. Therefore it does not enclose $q_{k'_{s}}$ as well as $q'_{k'_{s}}$. This implies $(X_{k', \T_k})_{s, s} = 0$.
	Take any $j$ with $1 \leq j < s$. The curve $\xi_s$ encloses precisely one puncture among $q'_1, \dots, q'_s$. Therefore, it also encloses precisely one puncture among $q_1, \dots, q_s$. Consequently, the curve $\xi_s$ cannot enclose the punctures $q_{j+1}$ and $q_{k'_j}$ simultaneously since $j+1 \leq s$ and $k'_j \leq s$. Hence, by formula (\ref{matrix}), we have $(X_{k', \T_k})_{j, s} = 0$. 
	
	Therefore, we have $(X_{k', \T_k})_{j, s} = 0$ for $1 \leq j \leq s$. This completes the proof.
\end{proof}
Theorem \ref{cyctree} immediately follows form Lemmas \ref{l1} and \ref{l2}.
\end{proof}
\begin{corollary} \label{gen}
	The abelian cycles $\{\hat{\A}_\T \; | \; \T \in \mathbf{T}_{g}^b\}$ form a basis of the group $\H_{g-2}(\PB_{g-1}, \Z)$. For any $\T \in \mathbf{T}_{g}$ we have $\hat{\A}_\T = \sum_{k \in \mathbf{K}_g} (-1)^{{g-2}\choose{2}} \det(X_{k, \T}) \hat{\A}_{\T_k}$.
\end{corollary}
\begin{proof}
	The result follows from Corollary \ref{cor3} and Theorem \ref{cyctree}. 
\end{proof}

\begin{proof}[Proof of Theorem \ref{main3}]
By Proposition \ref{th2} and the first part of Corollary \ref{gen} there is an isomorphism
$$\P_g \cong \H_{2g-3}(\Z^{g-1} \times \PB_{g-1}, \Z) \cong \H_{g-2}(\PB_{g-1}, \Z)$$
which maps $\A_{\T}$ to $\hat{\A}_\T$ for all $\T \in \mathbf{T}_{g}$. The theorem follows from the second part of Corollary \ref{gen}.
\end{proof}

\subsection{Relations} Let $\{\T_1, \T_2, \T_3\} \subseteq \mathbf{T}_{g}$ be a triple of trees. For $l = 1, 2, 3$ denote by $\xi^l_1, \dots, \xi^{l}_{g-2}$ the corresponding sets of curves given by Construction \ref{constr}. As before, the leaves of $\T_1, \T_2, \T_3$ (besides the root) are identified with the corresponding punctures and marked by $q_1, \dots, q_{g-1}$. One can check that the trees $\T_1, \T_2, \T_3$ form a cyclic triple if and only if after some permutations of the corresponding sets of curves the following conditions holds.

(a) There exists $s$ with $1\leq s \leq g-2$ such that $\xi_i^1= \xi_i^2 = \xi_i^3$ for $i \neq s$ and $1 \leq i \leq g-2$.

(b) There exists $t \neq s$ with $1 \leq t \leq g-2$ and pairwise disjoint nonempty subsets $B_1, B_2, B_3 \subset \{q_1, \dots, q_{g-1}\}$ such that the set of punctures enclosed by the curve $\xi_t^1= \xi_t^2 = \xi_t^3$ coincides with $B_1 \cup B_2 \cup B_3$.

(c) The set of punctures enclosed by $\xi_s^1$ coincides with $B_2 \cup B_3$.

(d) The set of punctures enclosed by $\xi_s^2$ coincides with $B_3 \cup B_1$.

(e) The set of punctures enclosed by $\xi_s^3$ coincides with $B_1 \cup B_2$.

\begin{lemma} \label{cyc}
	Let $\{\T_1, \T_2, \T_3\} \subseteq \mathbf{T}_{g}$ be a cyclic triple of trees. Then
	\begin{equation} \label{cyc1}
	\hat{\A}_{\T_{1}} + \hat{\A}_{\T_{2}} + \hat{\A}_{\T_{3}} = 0.
	\end{equation}
\end{lemma}
\begin{proof}
	It is suffices to prove that
	\begin{equation} \label{eq1}
	\langle W_k, \hat{\A}_{\T_{1}} + \hat{\A}_{\T_{2}} + \hat{\A}_{\T_{3}} \rangle = 0
	\end{equation}
	for all $k \in \mathbf{K}_g$. The formula (\ref{eq1}) is equivalent to
	\begin{equation} \label{eq2}
	\det(X_{k, \T_1}) + \det(X_{k, \T_2}) + \det(X_{k, \T_3}) = 0.
	\end{equation}
	We can assume that the conditions (a)-(e) hold. The matrices $X_{k, \T_1}, X_{k, \T_2}, X_{k, \T_3}$ coincide everywhere besides the $s^{\rm th}$ column. Therefore the left hand side of (\ref{eq2}) equals to the determinant of the matrix $Y$ defined as follows. The $s^{\rm th}$ column of $Y$ equal to the $s^{\rm th}$ column of the matrix $X_{k, \T_1} + X_{k, \T_2} + X_{k, \T_3}$ and all other columns equal to the correspondent columns of $X_{k, \T_1}$ (or, equivalently, $X_{k, \T_2}$ or $X_{k, \T_3}$). By the condition (3) we have
	$$(X_{k, \T_{1}})_{i, s}
	= \begin{cases}
	1 \mbox{, if } i \in B_{2} \cup B_{3}, \\
	0 \mbox{, otherwise,}
	\end{cases}$$
	$$(X_{k, \T_{2}})_{i, s}
	= \begin{cases}
	1 \mbox{, if } i \in B_{3} \cup B_{1}, \\
	0 \mbox{, otherwise,}
	\end{cases}$$
	$$(X_{k, \T_{3}})_{i, s}
	= \begin{cases}
	1 \mbox{, if } i \in B_{1} \cup B_{2}, \\
	0 \mbox{, otherwise.}
	\end{cases}$$
	Therefore,
	$$Y_{i, s} = (X_{k, \T_{1}} + X_{k, \T_{2}} + X_{k, \T_{3}})_{i, s} = \begin{cases}
	2 \mbox{, if } i \in B_{1} \cup B_{2} \cup B_{3}, \\
	0 \mbox{, otherwise.}
	\end{cases}$$
	By the condition (2) we have
	$$Y_{i, t} = (X_{k, \T_{1}})_{i, t} = 
	\begin{cases}
	1 \mbox{, if } i \in B_{1} \cup B_{2} \cup B_{3}, \\
	0 \mbox{, otherwise.}
	\end{cases}$$
	Therefore, the matrix $Y$ has two proportional columns, so $\det(Y) = 0$. This implies (\ref{eq2}).
\end{proof}

\begin{lemma} \label{cyc2}
	All relations between the abelian cycles $\{\hat{\A}_\T \; | \; \T \in \mathbf{T}_g\}$ follow from relations (\ref{cyc1}). 
\end{lemma}
\begin{proof}
	Consider the abelian cycle $\hat{\A}_{\T}$ for some $\T \in \mathbf{T}_g$. Corollary \ref{gen} implies that it suffices to decompose $\hat{\A}_\T$ into a linear combination of abelian cycles $\{\hat{\A}_{\T} \; | \; \T \in \mathbf{T}_{g}^b\}$ using relations (\ref{cyc1}). 
	
	Recall that a vertex of $\T$ of degree $3$ is called balanced if the paths from it to the descendant leaves with the two smallest numbers have no common edges.
	If all vertices of $\T$ are balanced it is nothing to prove. Otherwise take any nonbalanced vertex $v$ with the largest height (distance to the root) $h(v)$. Let $v_1$ and $v_2$ be its closest descendant and let $w$ be its closest ancestor. Without loss of generality we may assume that the paths from $v$ to the two descendant leaves with the smallest numbers start with the edge $(v, v_1)$. Let $u_1$ and $u_2$ be the closest descendants of $v_1$.
	\begin{figure}[h]
		\scalebox{1}{
			\begin{tikzpicture}
			\coordinate (w) at (2, 3);
			\coordinate (v) at (2, 2);
			\coordinate (v1) at (1, 1.5);
			\coordinate (v2) at (3,1);
			\coordinate (u1) at (0.5,0);
			\coordinate (u2) at (2,0);
			\fill[black]  (w) circle [radius=3pt];
			\fill[black]  (v) circle [radius=3pt];
			\fill[black]  (v1) circle [radius=3pt];
			\fill[black]  (v2) circle [radius=3pt];
			\fill[black]  (u1) circle [radius=3pt];
			\fill[black]  (u2) circle [radius=3pt];
			\draw{(w) -- (v)};
			\draw{(v) -- (v1)};
			\draw{(v) -- (v2)};
			\draw{(v1) -- (u1)};
			\draw{(v1) -- (u2)};
			
			\node[scale = 1][above] at (w) {$w$};
			\node[scale = 1][right] at (v) {$v$};
			\node[scale = 1][left] at (v1) {$v_1$};
			\node[scale = 1][right] at (v2) {$v_2$};
			\node[scale = 1][below] at (u1) {$u_1$};
			\node[scale = 1][below] at (u2) {$u_2$};
			
			\coordinate (w) at (6, 3);
			\coordinate (v) at (6, 2);
			\coordinate (v1) at (5, 1.5);
			\coordinate (v2) at (7,1);
			\coordinate (u1) at (4.5,0);
			\coordinate (u2) at (6,0);
			\fill[black]  (w) circle [radius=3pt];
			\fill[black]  (v) circle [radius=3pt];
			\fill[black]  (v1) circle [radius=3pt];
			\fill[black]  (v2) circle [radius=3pt];
			\fill[black]  (u1) circle [radius=3pt];
			\fill[black]  (u2) circle [radius=3pt];
			\draw{(w) -- (v)};
			\draw{(v) -- (v1)};
			\draw{(v1) -- (v2)};
			\draw{(v1) -- (u1)};
			\draw{(v) -- (u2)};
			
			\node[scale = 1][above] at (w) {$w$};
			\node[scale = 1][right] at (v) {$v$};
			\node[scale = 1][left] at (v1) {$v_1$};
			\node[scale = 1][right] at (v2) {$v_2$};
			\node[scale = 1][below] at (u1) {$u_1$};
			\node[scale = 1][below] at (u2) {$u_2$};
			
			\coordinate (w) at (10, 3);
			\coordinate (v) at (10, 2);
			\coordinate (v1) at (9, 1.5);
			\coordinate (v2) at (11,1);
			\coordinate (u1) at (8.5,0);
			\coordinate (u2) at (10,0);
			\fill[black]  (w) circle [radius=3pt];
			\fill[black]  (v) circle [radius=3pt];
			\fill[black]  (v1) circle [radius=3pt];
			\fill[black]  (v2) circle [radius=3pt];
			\fill[black]  (u1) circle [radius=3pt];
			\fill[black]  (u2) circle [radius=3pt];
			\draw{(w) -- (v)};
			\draw{(v) -- (v1)};
			\draw{(v) -- (u1)};
			\draw{(v1) -- (v2)};
			\draw{(v1) -- (u2)};
			
			\node[scale = 1][above] at (w) {$w$};
			\node[scale = 1][right] at (v) {$v$};
			\node[scale = 1][left] at (v1) {$v_1$};
			\node[scale = 1][right] at (v2) {$v_2$};
			\node[scale = 1][below] at (u1) {$u_1$};
			\node[scale = 1][below] at (u2) {$u_2$};

				\end{tikzpicture}}
		\caption{The trees $\T$, $\T'$ and $\T''$.}
	\label{cyclic2}
\end{figure}
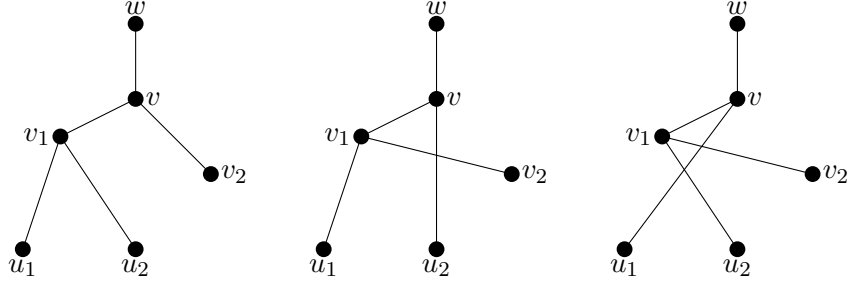

Consider the trees $\T'$ and $\T''$ that differ from $\T$ as shown in Fig. \ref{cyclic2}. The triple $\{\T, \T', \T''\}$ is cyclic, so $\hat{\A}_{\T} = - \hat{\A}_{\T'} - \hat{\A}_{\T''}$. Note that the vertex $v_1$ is balanced in $\T$, therefore the vertex $v$ is balanced in $\T'$ and $\T''$. Consequently, $\T'$ and $\T''$ have less number of nonbalanced vertices of height $h(v)$ and no nonbalanced vertices of greater height. Repeating this operation we decompose $\hat{\A}_\T$ into a linear combination of abelian cycles $\{\hat{\A}_{\T} \; | \; \T \in \mathbf{T}_{g}^b\}$ using relations (\ref{cyc1}). This completes the proof.
\end{proof}

\begin{proof}[Proof of Proposition \ref{th1}] By Proposition \ref{th2} and the first part of Corollary \ref{gen} there is an isomorphism
$$\P_g \cong \H_{2g-3}(\Z^{g-1} \times \PB_{g-1}, \Z) \cong \H_{g-2}(\PB_{g-1}, \Z)$$
which maps $\A_{\T}$ to $\hat{\A}_\T$ for all $\T \in \mathbf{T}_{g}$. The abelian cycles $\{\hat{\A}_\T \; | \; \T \in \mathbf{T}_g\}$ generate the group $\H_{g-2}(\PB_{g-1}, \Z)$, therefore the required assertion follows from Lemmas \ref{cyc} and \ref{cyc2}.
\end{proof}

\section{Complex of cycles and the spectral sequence}

Consider the commutative diagram
\begin{equation*} \label{sub}
\scalebox{1}{
	\xymatrix{	
		1 \ar[r] & \wedge^{3} \H / \H \ar[r] & \mathcal{G}_g \ar[r]^(0.4)p &  \Sp(2g, \Z) \ar[r]   & 1\\
		&& \mathcal{H}_g = \SL(2, \Z)^{\times g} \rtimes S_{g} \ar@{^{(}->}[u] \ar@{^{(}->}[ru] &&
}}
\end{equation*}

Let us choose elements $h_{1} = 1, h_{2}, h_{3},  \dots \in \mathcal{G}_g$ such that $1= p(h_1), p(h_2),  p(h_3), \dots \in \Sp(2g, \Z)$ are representatives of all left cosets $\Sp(2g, \Z) / \mathcal{H}_g$. Let $\hat{h}_1, \hat{h}_2, \hat{h}_3 \dots \in \Mod(\S_{g})$ be mapping classes that go to $h_1, h_2, h_3, \dots$ under the natural surjective homomorphism $\Mod(\S_{g}) \twoheadrightarrow \mathcal{G}_g$.

It is convenient to denote by $\U_g$ the abelian group $\Lambda^3 \H / \H$ with multiplicative notation. For each $u \in \U_g$ let $\hat{u} \in \I_g$ be mapping class that go to $u$ under the Johnson homomorphism $\tau: \I_g \to \U_g$.
Let $f_{1} = 1, f_{2}, f_{3},  \dots \in \mathcal{G}_g$ be representatives of all left cosets $\mathcal{G}_g / \mathcal{H}_g$.
Let $\hat{f}_1, \hat{f}_2, \hat{f}_3, \dots \in \Mod(\S_{g})$ be mapping classes that go to $f_1, f_2, f_3, \dots$ under the homomorphism $\Mod(\S_{g}) \twoheadrightarrow \mathcal{G}_g$. 
For any $s \in \mathbb{N}$ the element $f_s$ can be uniquely decomposed as $f_s = u \cdot h_r$ for some $u \in \U_g$ and $r \in \mathbb{N}$. We can choose $\hat{f}_s$ such that $\hat{f}_s = \hat{u} \cdot \hat{h}_r$.

For each $r\in \mathbb{N}$ let us denote by $G_r$ the subgroup of $\H_{2g-3}(\K_g, \Z)$ generated by the images of homomorphisms
\begin{equation} \label{eq5}
\H_{2g-3}(\Stab_{\K_g}(\hat{u} \cdot \hat{h}_r \cdot N), \Z) \rightarrow \H_{2g-3}(\K_g, \Z), \; \; \; u \in \U_g.
\end{equation}

In this section we prove the following result.
\begin{lemma} \label{lemma}
	The inclusions
	\begin{equation*}\label{inc}
	G_r \hookrightarrow \H_{2g-3}(\K_g, \Z), \;\;\; r \in \mathbb{N}
	\end{equation*}
	induce an injective homomorphism
	\begin{equation*} \label{eq3}
	\bigoplus_{r \in \mathbb{N}} G_r \hookrightarrow \H_{2g-3}(\K_g, \Z).
	\end{equation*}
\end{lemma}
In our proof we follow ideas of \cite{Gaifullin_J}.
\subsection{Complex of cycles} 

Bestvina, Bux, and Margalit \cite{Bestvina} constructed a contractible $CW$-complex $\B_g$ called the \textit{complex of cycles} on which the Johnson kernel acts without rotations. ''Without rotations'' means that if an element  $h \in \K_g$ stabilizes a cell $\sigma$ setwise, then $h$ stabilizes $\sigma$ pointwise. Let us recall the construction of $\B_g$. More details can be found in \cite{Bestvina, Hatcher, Gaifullin_T, Gaifullin_T3}.

Let us denote by $\C$ the set of all isotopy classes of oriented non-separating simple closed curves on $\S_{g}$. Fix any nonzero element $x \in \H$. The construction of $\B_g = \B_g(x)$ depends on the choice of the homology class $x$, however the $CW$-complexes $\B_g(x)$ are pairwise homeomorphic for different $x$. 

A \textit{basic 1-cycle} for a homology class $x$ is a formal linear combination $\gamma = \sum_{i=1}^n k_i  \gamma_i$, where $\gamma_i \in \C$ and $k_i \in \mathbb{N}$, satisfying the following properties:

(1) the homology classes $[\gamma_1], \dots, [\gamma_n]$ are linearly independent,

(2) $\sum_{i=1}^n k_i [\gamma_i] = x$,

(3) the isotopy classes $\gamma_1, \dots, \gamma_g$ contain pairwise disjoint representatives.

The oriented multicurve $\gamma_{1} \cup \dots \cup \gamma_{g}$ is called the \textit{support} of $\gamma$. 

Let us denote by $\M(x)$ the set of oriented multicurves $M = \gamma_{1} \cup \dots \cup \gamma_s$ satisfying the following properties:

(i) no nontrivial linear combination of the homology classes $[\gamma_1], \dots, [\gamma_s]$ with nonnegative coefficients equals zero,

(ii) for each $1 \leq i \leq s$ there exists a basic 1-cycle for $x$ whose support is contained in $M$ and contains $\gamma_i$.

For each $M \in \M(x)$ let us denote by $P_M \subset \mathbb{R}_{\geq 0}^{\C}$ the convex hull of the basic 1-cycles supported in $M$.  We have that $P_M$ is a convex polytope. By definition the complex of cycles is the regular $CW$-complex given by $\B_g(x) = \cup_{M \in \M(x)} P_M$. Denote by $\M_0(x) \subseteq \M(x)$ the set of supports of basic 1-cycles for $x$. Then $\{P_M \; | \; M \in \M_0(x)\}$ is the set of 0-cells of $\B_g(x)$.  

\begin{theorem} \cite[Theorem E]{Bestvina}\label{contrcyc}
	Let $g \geq 1$ and $0 \neq x \in \H_{1}(\S_g, \Z)$. Then $\B_g(x)$ is contractible.
\end{theorem}

\subsection{The spectral sequence}

Suppose that a group $G$ acts cellularly and without rotations on a contractible $CW$-complex $X$. 
Let $C_*(X, \Z)$ be the cellular chain complex of $X$ and $\R_*$ be a projective resolution of $\Z$ over $\Z G$. Consider the double complex  $B_{p, q} = C_p(X, \Z) \otimes_G \R_q$ with the filtration by columns. The corresponding spectral sequence (see (7.7) in \cite[Section VII.7]{Brown}) has the form
\begin{equation} \label{spec_sec}
E_{p, q}^1 \cong \bigoplus_{\sigma \in \X_p}\H_q(\Stab_G (\sigma), \Z) \Rightarrow \H_{p+q}(G, \Z),
\end{equation}
where $\X_p$ is a set containing exactly one representative in each $G$-orbit of $p$-cells of $X$. Let us remark that for an arbitrary $CW$-complex $X$ the spectral sequence (\ref{spec_sec}) converges to the equivariant homology $\H^G_{p+q}(X, \Z)$. So for a contractable $CW$-complex $X$ it converges to $\H^G_{p+q}(X, \Z) \cong \H_{p+q}(G, \Z)$.

Now let $E_{*,*}^*$ be the spectral sequence (\ref{spec_sec}) for the action of $\K_g$ on $\B_g(x)$ for some $0 \neq x \in \H_{1}(\S_g, \Z)$. The fact that $\K_g$ acts on $\B_g$ without rotations follows from a result of Ivanov \cite[Theorem 1.2]{Ivanov}: if an element $h \in \I_g$ stabilises a multicurve $M$ then $h$ stabilises each component of $M$. Bestvina, Bux and Margalit proved \cite[Proposition 6.2]{Bestvina} that for each cell $\sigma \in \B_g(x)$ we have
\begin{equation*} \label{ineq}
\dim(\sigma) + \cd(\Stab_{\K_g}(\sigma)) \leq 2g-3.
\end{equation*}

\begin{figure}[h]
	\scalebox{1}{
		\begin{tikzpicture}
			\path[fill=green] (0, 7) -- (1, 7) -- (1, 8) -- (0, 8) -- cycle;
			
			\draw[thick, ->] (0, 0) to (9, 0);
			\draw[thick, ->] (0, 0) to (0, 9);
			
			\draw[thick] (1, 0) to (1, 3.2);
			\draw[thick] (2, 0) to (2, 3.2);
			\draw[thick] (3, 0) to (3, 3.2);
			
			\draw[thick] (0, 1) to (3.2, 1);
			\draw[thick] (0, 2) to (3.2, 2);
			\draw[thick] (0, 3) to (3.2, 3);
			
			\draw[thick] (5, 0) to (5, 3.2);
			\draw[thick] (6, 0) to (6, 3);
			\draw[thick] (7, 0) to (7, 2);
			\draw[thick] (8, 0) to (8, 1);
			
			\draw[thick] (4.8, 3) to (6, 3);
			\draw[thick] (4.8, 2) to (7, 2);
			\draw[thick] (4.8, 1) to (8, 1);

			\draw[thick] (0, 5) to (3.2, 5);
			\draw[thick] (0, 6) to (3, 6);
			\draw[thick] (0, 7) to (2, 7);
			\draw[thick] (0, 8) to (1, 8);
			
			\draw[thick] (3, 4.8) to (3, 6);
			\draw[thick] (2, 4.8) to (2, 7);
			\draw[thick] (1, 4.8) to (1, 8);
			
			\draw[thick, dotted] (3.2, 1) to (4.8, 1);
			\draw[thick, dotted] (3.2, 2) to (4.8, 2);
			\draw[thick, dotted] (3.2, 3) to (4.8, 3);
			
			\draw[thick, dotted] (1, 3.2) to (1, 4.8);
			\draw[thick, dotted] (2, 3.2) to (2, 4.8);
			\draw[thick, dotted] (3, 3.2) to (3, 4.8);
			
			\draw[thick, dotted] (3.2, 4.8) to (4.8, 3.2);

			\node[scale = 1][below] at (8.5, 0) {$p$};
			\node[scale = 1][below] at (7.5, 0) {$2g-3$};
			\node[scale = 1][below] at (0.5, 0) {$0$};
			\node[scale = 1][below] at (1.5, 0) {$1$};
			\node[scale = 1][below] at (2.5, 0) {$2$};
			
			\node[scale = 1][left] at (0, 0.5) {$0$};
			\node[scale = 1][left] at (0, 1.5) {$1$};
			\node[scale = 1][left] at (0, 2.5) {$2$};
			\node[scale = 1][left] at (0, 8.5) {$q$};
			\node[scale = 1][left] at (0, 7.5) {$2g-3$};
			\node[scale = 1][left] at (0, 6.5) {$2g-4$};
			\node[scale = 1][left] at (0, 5.5) {$2g-5$};

			\draw[red, thick, ->] (1.5, 7.65) to (0.5, 7.65);
			\draw[red, thick, ->] (2.5, 6.5) to (0.5, 7.5);
			\draw[red, thick, ->] (3.5, 5.35) to (0.5, 7.35);
			
			\draw[red, thick, dotted] (4, 5) to (5.5, 3.5);

	\end{tikzpicture}}
	\caption{ }
	\label{spectral}
\end{figure}

This immediately implies $E^1_{p, q} = 0$ for $p+q > 2g-3$. Hence all differentials $d^1, d^2, \dots$ to the group $E^1_{0, 2g-3}$ are trivial (see Fig. \ref{spectral}, the group $E_{0, 2g-3}^1$ is shown in green), so $E^1_{0, 2g-3} = E^{\infty}_{0, 2g-3}$. Therefore we have the following result.
\begin{prop} \cite[Proposition 3.2]{Gaifullin_J} \label{prop1}
	Let $\mathfrak{M} \subseteq \mathcal{M}_{0}(x)$ be a subset consisting of oriented multicurves from pairwise different $\K_g$-orbits. Then the inclusions $\Stab_{\K_g}(M) \subseteq \K_g$, where $M \in \mathfrak{M}$, induce an injective homomorphism
	\begin{equation*} \label{c-l}
	\bigoplus_{M \in \mathfrak{M}} \H_{2g-3}(\Stab_{\K_g}(M), \Z) \hookrightarrow \H_{2g-3}(\K_g), \Z).
	\end{equation*}
\end{prop}

\begin{proof}[Proof of Lemma \ref{lemma}]
Denote by $X_{r, i} \subset \S_{g}$ the one-punctured torus bounded by $\hat{h}_r\delta_{i}$ and $V_{r, i} = \H_{1}(X_{r, i}, \Z) \subset \H$. Then for each $r$ we have the symplectic splittings $\H = \oplus_i V_{r, i}$.
Denote this unordered splitting by $\V_{r} = \{V_{r, 1}, \dots, V_{r, g}\}$. Since $\mathcal{H}_g$ is the stabiliser of $\V_1$ in $\Sp(2g, \Z)$, it follows that $\V_{r}$ are pairwise distinct.

Assume the converse to the statement of Lemma \ref{lemma} and consider a nontrivial linear relation
\begin{equation} \label{relation}
\sum_{r=1}^k \lambda_r \theta_r = 0, \; \; \lambda_r \in \Z, \; \theta_r \in G_r.
\end{equation}
For any homology class $x \in \H_{1}(\S_g, \Z)$ and for any $1 \leq r \leq k$ we have a unique decomposition
\begin{equation*}\label{dec}
x = \sum_{i=1}^g x_{r, i}, \; \; x_{r, i} \in V_{r, i}.
\end{equation*}
The following result is proved in \cite{Gaifullin_J}.
\begin{prop}\cite[Lemma 4.5]{Gaifullin_J} \label{prop_dec}
	There is a homology class $x \in \H$ such that
	
	(1) all homology classes $x_{r, i}$ are nonzero, $1 \leq r \leq k$, $1 \leq i \leq g$,
	
	(2) for all $1\leq p \neq q \leq k$ we have $\{x_{p, 1}, \dots, x_{p, g}\} \neq \{x_{q, 1}, \dots, x_{q, g}\}$ as unordered sets.
\end{prop}
Take any $x \in \H$ satisfying the conditions of Proposition \ref{prop_dec}.
For any $1 \leq r \leq k$ and $1 \leq i \leq g$ we have $x_{r, i} = n_{r, i} a_{r, i}$ where $a_{r, i} \in \H$ is primitive and $n_{r, i} \in \mathbb{N}$. Let us check that for all $1\leq p \neq q \leq k$ we have $\{a_{p, 1}, \dots, a_{p, g}\} \neq \{a_{q, 1}, \dots, a_{q, g}\}$ as unordered sets. Indeed, assume that there is a permutation $\pi \in S_g$ with $a_{p, i} = a_{q, \pi(i)}$. Therefore we have
\begin{equation} \label{eq4}
\sum_{i=1}^{g} (n_{p, i} - n_{p, \pi(i)}) a_{p, i} = 0.
\end{equation}
Since $a_{p, 1}, \dots, a_{p_g}$ are linearly independent (\ref{eq4}) implies $n_{p, i} = n_{p, \pi(i)}$ for all $1 \leq i \leq g$. Hence $x_{p, i} = x_{p, \pi(i)}$ for all $1 \leq i \leq g$, which contradicts to the condition (2) of Proposition \ref{prop_dec}.

For any $1 \leq r \leq k$ and $1 \leq i \leq g$ let $\alpha_{r, i}$ be a simple curve on $X_{r, i}$ with $[\alpha_{r, i}] = a_{r, i} \in \H$. Consider the oriented multicurve $A_r = \alpha_{r, 1} \cup \dots \cup \alpha_{r, g}$. By construction we have $A_r \in \M_0(x)$.

Proposition \ref{th1} implies that the group  $\H_{2g-3}(\Stab_{\K_g}(\hat{u} \cdot \hat{h}_r \cdot N), \Z)$ is generated by the primitive abelian cycles $\{\hat{u} \cdot \hat{h}_r \cdot \A_\T \; | \; \T \in \mathbf{T}_g\}$, therefore
for each  $u \in \U_g$ the homomorphisms (\ref{eq5}) can be decomposed as
\begin{equation} \label{eq6}
\H_{2g-3}(\Stab_{\K_g}(\hat{u} \cdot \hat{h}_r \cdot N), \Z) \rightarrow  \H_{2g-3}(\Stab_{\K_g}(A_r), \Z) \rightarrow \H_{2g-3}(\K_g, \Z).
\end{equation}
Consequently, there exists $\theta'_r \in \H_{2g-3}(\Stab_{\K_g}(A_r), \Z)$ which maps to $\theta_r$ under the second homomorphism in (\ref{eq6}).

Proposition \ref{prop1} implies that the inclusions $\Stab_{\K_g}(A_r) \subseteq \K_g$, $r \in \mathbb{N}$ induce the injective homomorhpism
\begin{equation*} \label{c-l2}
\bigoplus_{r \in \mathbb{N}} \H_{2g-3}(\Stab_{\K_g}(A_r), \Z) \hookrightarrow \H_{2g-3}(\K_g, \Z)
\end{equation*}
Consequently, (\ref{relation}) implies that we have $\sum_{r=1}^k \lambda_r \theta'_r = 0$ as an element of the direct sum $\bigoplus_{r \in \mathbb{N}} \H_{2g-3}(\Stab_{\K_g}(A_r), \Z)$.
Therefore, $\lambda_r = 0$ for all $r$, which gives a contradiction. \end{proof}

\section{Proof of Proposition \ref{th2}}

In this section we prove the following lemma, which implies Proposition \ref{th2}. Recall that $G_r$ is the subgroup of $\H_{2g-3}(\K_g, \Z)$ generated by the images of homomorphisms
\begin{equation*} 
\H_{2g-3}(\Stab_{\K_g}(\hat{u} \cdot \hat{h}_r \cdot N), \Z) \rightarrow \H_{2g-3}(\K_g, \Z), \; \; \; u \in \U_g.
\end{equation*}
\begin{lemma} \label{4lemma}
	Let $r \in \mathbb{N}$. Then the inclusions
	\begin{equation*} \label{inclu}
	\Stab_{\K_g}(\hat{u} \cdot \hat{h}_r \cdot N) \hookrightarrow \K_g, \; \; \; u \in \U_g
	\end{equation*}
	induce an injective homomorphism
	\begin{equation*} \label{eq7}
	\bigoplus_{u\in \U_g} \H_{2g-3}(\Stab_{\K_g}(\hat{u} \cdot \hat{h}_r \cdot N), \Z) \hookrightarrow G_r.
	\end{equation*}
\end{lemma}

\begin{proof}[Proof of Proposition \ref{th2}]
We can proof Proposition \ref{th2} for an arbitrary choice of $\hat{f}_s$, so we can assume that $\hat{f}_s = \hat{u} \cdot \hat{h}_r$ fore some $u \in \U_g$ and $r \in \mathbb{N}$.
Combining Lemmas \ref{lemma} and \ref{4lemma}, we obtain
\begin{equation} \label{eq8}
\bigoplus_{r \in \mathbb{N}} \bigoplus_{u\in \U_g} \H_{2g-3}(\Stab_{\K_g}(\hat{u} \cdot \hat{h}_r \cdot N), \Z) \hookrightarrow \bigoplus_{r \in \mathbb{N}}G_r \hookrightarrow \H_{2g-3}(\K_g, \Z).
\end{equation}
Then the sets $\{ \; | \; s \in \mathbb{N}\}$ and $\{\hat{u} \cdot \hat{h}_r \; | \; u \in \U_g, \; r \in \mathbb{N}\}$ coincide, so (\ref{eq8}) implies (\ref{eq0}).
\end{proof}

To prove Lemma \ref{4lemma} we need to construct a new $CW$-complex, which will be called the \textit{complex of relative cycles}. The idea is to introduce an analogue of $\B_g$ that makes sense for a sphere (i.e. $g=0$ case) with punctures.

\subsection{Complex of relative cycles}

Recall that by $\S_{0, 2g}$ we denote a sphere with $2g$ punctures. In order to construct the complex of relative cycles $\B_{0, 2g}$ we need to split the punctures into two disjoint sets: $P = \{p_1, \dots, p_g\}$ and $Q = \{q_1, \dots, q_g\}$.

By an  \textit{arc} on $\S_{0, 2g}$ we mean an embedded oriented curve with endpoints at punctures. By a \textit{multiarc} we mean a disjoint union of arcs (common endpoints are allowed). We always consider arcs and multiarcs up to an isotopy.

Denote by $\D$ the set of isotopy classes of arcs starting at a point in $P$ and finishing at a point in $Q$. \textit{Relative basic 1-cycle} is a formal sum $\gamma = \gamma_{1}+\dots+\gamma_{g}$ where $\gamma_i \in \D$ such that

(1) $\partial (\sum_{i=1}^{g}\gamma_{i}) = \sum_{i=1}^{g}(q_{i} - p_{i})$,

(2) the isotopy classes $\gamma_1, \dots, \gamma_g$ contain pairwise disjoint representatives.

The multiarc $\gamma_{1} \cup \dots \cup \gamma_{g}$ is called the \textit{support} of $\gamma$. 

Denote by $\mathcal{L}$ the set of multiarcs $L = \gamma_{1} \cup \dots \cup \gamma_n$ (for arbitrary $n$) satisfying the following property:

(i) for each $1 \leq i \leq s$ there exists a relative basic 1-cycle, whose support is contained in $L$ and contains $\gamma_i$.

For each $L \in \mathcal{L}$ we denote by $P_L \subset \mathbb{R}_{\geq 0}^{\D}$ the convex hull of all relative basic 1-cycles supported in $L$. We have that $P_L$ is a convex polytope. By definition the complex of relative cycles is the regular $CW$-complex given by $\B_{0, 2g} = \cup_{L \in \mathcal{L}} P_L$. Denote by $\mathcal{L}_0 \subseteq \mathcal{L}$ the set of supports of all relative basic 1-cycles. Then $\{P_L \; | \; L \in \mathcal{L}_0\}$ is the set of 0-cells of $\B_{0, 2g}$.

\begin{remark} \label{rem}
	By construction $\B_{0, 2g}$ is the subset of $\mathbb{R}_{\geq}^\D$ consisting of the points (formal sums) $\sum_{i=1}^n k_i \gamma_i$ where $\gamma_i \in \D$ and $k_i \in \mathbb{R}_{\geq 0}$ satisfying the following conditions:
	
	(1) $\partial (\sum_{i=1}^{n}k_i \gamma_{i}) = \sum_{i=1}^{g}(q_{i} - p_{i})$.
	
	(2) the isotopy classes $\gamma_1, \dots, \gamma_n$ contain pairwise disjoint representatives.
\end{remark}

\subsection{Contractability}

\begin{theorem} \label{contr}
	Let $g \geq 1$. Then $\B_{0, 2g}$ is contractible.
\end{theorem}

In our proof we follow ideas of \cite[Section 5]{Bestvina}.
Let us define an auxiliary complex $\widetilde{\B}_{0, 2g}$. Denote by $\widetilde{\D}$ the union of $\D$ and the set consisting of the isotopy classes of all oriented simple closed curves on $\S_{0, 2g}$ (including contractible curves). Let us define $\widetilde{\B}_{0, 2g}$ as the subset of $\mathbb{R}_{\geq 0}^{\widetilde{\D}}$ consisting of all points (formal sums) $\sum_{i=1}^n k_i \gamma_i$ where $\gamma_i \in \widetilde{D}$ and $k_i \in \mathbb{R}_{\geq 0}$ satisfying the following conditions:

(1) $\partial (\sum_{i=1}^{n}k_i \gamma_{i}) = \sum_{i=1}^{g}(q_{i} - p_{i})$.

(2) the isotopy classes $\gamma_1, \dots, \gamma_n$ contain pairwise disjoint representatives.

Remark \ref{rem} implies that $\B_{0, 2g} \subseteq \widetilde{\B}_{0, 2g}$. Denote by $\Drain: \widetilde{\B}_{0, 2g} \to \B_{0, 2g}$ the retraction induced by the natural projection $\mathbb{R}_{\geq 0}^{\widetilde{\D}} \to \mathbb{R}_{\geq 0}^\D$.

Let $d$ and $d'$ be two points of $\B_{0, 2g} \subseteq \mathbb{R}_{\geq 0}^{\D}$ and $t \in [0, 1]$. The point $c = td + (1-t)d' \in \mathbb{R}_{\geq 0}^{\D}$ mat not belong to $\B_{0, 2g}$, because the arcs can have intersection points. We now explain how to do surgery to convert $c$ into a point $\Surger(c) \in \widetilde{\B}_{0, 2g} \subseteq \mathbb{R}_{\geq 0}^{\widetilde{\D}}$, which is canonical up to isotopy.

Let $c = \sum_{i=1}^n k_i c_i$ where $c_i \in \D$ are in minimal position and $k_i \in \mathbb{R}_{\geq 0}$.  
We have $\partial (\sum_{i=1}^{n}k_i c_{i}) = \sum_{i=1}^{g}(q_{i} - p_{i})$.
Now it is convenient to replace the punctures $p_1, \dots, p_g, q_1, \dots, q_g$ by closed disks $P_1, \dots, P_g, Q_1, \dots, Q_g$. We thicken each $c_i$ to a rectangle $R_i = [0, 1] \times [0, k_i]$ of width $k_i$ with coordinates $x_i \in [0, 1]$ and $t_i \in [0, k_i]$ such that the curves $t_i = const$ for different $i$ are transversal to each other. We assume that the sides of $R_i$ given by $x=0$ and $x=1$ are subsets of $\partial P_a$ and $\partial Q_b$ respectively, where $\partial c_i = q_b - p_a$.  

For a path $\alpha: [0, 1] \to \S_{0, 2g}$ define $\mu_i(\alpha) = \int_{\alpha} dt_i$ and $\mu(\alpha) = \sum_{i = 1}^n \mu_i(\alpha)$. Here we assume that $dt_i = 0$ outside $R_i$. Let us fix an arbitrary point $y_0 \in \S_{0, 2g}$. For each point $y \in \S_{0, 2g}$ choose a path $\alpha_y$ from $y_0$ to $y$. Consider the map $\phi: \S_{0, 2g} \to S^1 = \mathbb{R}/ \Z$ given by $\phi(y) = \mu(\alpha_y) \mod 1$.

Let us check that the map $\phi$ is well defined. We have that $\phi(x)$ depends only on the homotopy class of $\alpha_x$. Therefore it suffices to check that $\mu(\partial P_i)\in \Z$ and $\mu(\partial Q_i) \in \Z$ for all $i$. This follows from the fact that $\partial (\sum_{i=1}^{n}k_i c_{i}) = \sum_{i=1}^{g}(q_{i} - p_{i})$.

The set of zeros of $d\phi$ is precisely $\S_{0, 2g} \setminus \cup_{i=1}^g R_i$, that is a finite disjoint union of connected open sets. Therefore the map $\phi$ has a finite number of critical values separating $S^1$ into a finite number of intervals $w_1, \dots, w_l$. For any $1 \leq j \leq l$ take any point $y_j \in w_j$. The preimage $\eta_j = \phi^{-1}(y_j) \subset \S_{0, 2g}$ is a smooth 1-dimensional oriented submanifold, where the orientation on $\eta_j$ is defined such that at each point of $\eta_j$ the vector $\frac{\partial}{\partial t_i}$ and the positive tangent vector to $\eta_s$ from a positive basis of the tangent space to the sphere. Moreover, $\eta_1, \dots, \eta_l$ are pairwise disjoint. Define $\Surger(c)$ as the formal sum $\sum_{j=1}^l |w_j| \eta_j$. 

We claim that $\Surger(c) \in \mathbb{R}_{\geq 0}^{\widetilde{\D}}$. It suffices to show that each connected component of $\eta_j$ is either closed or its initial point belongs to $\partial P_a$ for some $a$ and its terminal point belongs to $\partial Q_b$ for some $b$. This follows from the orientation argument. Indeed, for all $i$ the restrictions $\phi|_{\partial P_i}$ and $\phi|_{\partial Q_i}$ have degrees $-1$ and $1$ respectively. Hence $\phi|_{\partial P_i}$ can only contain initial point of components of $\eta_j$, while $\phi|_{\partial Q_i}$ can only contain terminal point of components of $\eta_j$. Consequently, no component of $\Surger(c)$ connects $\partial P_a$ with $\partial P_b$ or $\partial Q_a$ with $\partial Q_b$. Moreover, since the restrictions $\phi|_{\partial P_i}$ and $\phi|_{\partial Q_i}$ have degrees $-1$ and $1$ respectively we obtain $\partial (\Surger(c)) = \sum_{i=1}^{g}(q_{i} - p_{i})$, so $\Surger(c) \in \widetilde{\B}_{0, 2g}$.

\begin{proof}[Proof of Theorem \ref{contr}] Take a point $c \in \B_{0, 2g}$. Then the map
$$d \mapsto \Drain(\Surger(tc+(1-t)d))$$
is a deformation retraction from $\B_{0, 2g}$ to the point $c$.
\end{proof}

\subsection{Stabilizer dimensions}

\begin{prop}
	The group $\PMod(\S_{0, 2g})$ acts on $\B_{0, 2g}$ without rotations.
\end{prop}
\begin{proof}
	Assume the converse and consider an element $\phi \in \PMod(\S_{0, 2g})$ and a cell corresponding to a multiarc $\gamma = \gamma_1 \cup \dots \cup \gamma_s$ such that $\phi(\gamma_i) = \gamma_{\pi(i)}$ for a nontrivial permutation $\pi$. 
	Without loss of generality can assume that there exist arcs $\gamma_1, \gamma_2, \gamma_3$ from $p \in P$ to $q \in Q$ satisfying $\gamma_1 \neq \gamma_2$ and $\gamma_2 \neq \gamma_3$ (possibly $\gamma_1 = \gamma_3$), such that $\phi(\gamma_1) = \gamma_2$ and $\phi(\gamma_2) = \gamma_3$. 
	Denote by $W_1 \subset \S_{0, 2g}$ and $W_2 \subset \S_{0, 2g}$ the subsurfaces bounded by the loops $\gamma_1 \overline{\gamma}_2$ and $\gamma_2 \overline{\gamma}_3$ respectively (by $\overline{\gamma}_i$ we denote the arc $\gamma_i$ with opposite direction). We assume that $W_1$ and $W_2$ located on the left side of $\gamma_1 \overline{\gamma}_2$ and $\gamma_2 \overline{\gamma}_3$ respectively.
	
	By construction of $\B_{0, 2g}$ we see that $\gamma_1$ is not isotopic to $\gamma_2$, so $W_1$ contains a nonempty set of punctures $\varnothing \neq Z_1 \subset P \sqcup Q$. Define $\varnothing \neq Z_2 \subset P \sqcup Q$ in the similar way. Since  $\gamma_2$ separates $W_1$ from $W_2$ we have $Z_1 \neq Z_2$. The map $f$ preserves the orientation, therefore $f(W_1) = W_2$ and so $f(Z_1) = Z_2$. However, $f \in \PMod(\S_{0, 2g})$ preserves the punctures,so we come to a contradiction. 
\end{proof}

\begin{theorem} \label{stabdim}
	Let $\sigma$ be a cell of $\B_{0, 2g}$. Then
	\begin{equation*}\label{ineqdim}
	\dim(\sigma) + \cd(\Stab_{\PMod(\S_{0, 2g})}(\sigma)) \leq 2g-3.
	\end{equation*}
\end{theorem}
\begin{proof}
	The cell $\sigma$ is given by a multiarc $\gamma_1 \cup \dots \cup \gamma_E$. Consider the planar graph $\Upsilon$ on the sphere with the vertices $p_1, \dots, p_g, q_1, \dots, q_g$ and the edges $\gamma_1, \dots, \gamma_E$. It is convenient for us to denote the number of vertices by $V = 2g$. Also let us denote by $C$ number of the connected components of $\Upsilon$ and by $F$ number of its faces (i.e. number of connected components of $\S_{0, 2g} \setminus \Upsilon$). 
\begin{lemma} \label{dim}
	\begin{equation} \label{dim1}
	\dim(\sigma) = \dim(\H_{1}(\Upsilon, \mathbb{R})) = E - V + C.
	\end{equation}
\end{lemma}
\begin{proof}
	The condition $\partial (\sum_{i=1}^{E} k_i \gamma_i) = \sum_{i=1}^{g}(q_{i} - p_{i})$ is a nonhomogeneous  system of linear equation in $\mathbb{R}^{E}$. The affine space of its solutions has the same dimension as the space of solutions of the homogeneous system $\partial (\sum_{i=1}^{E} k_i \gamma_i) = 0$. This space is precisely $\H_{1}(\Upsilon, \mathbb{R})$. The cell $\sigma$ is given by the intersection of this affine space with $\mathbb{R}_{\geq 0}^E$. Condition (i) in the construction of $\B_{0, 2g}$ implies that $\sigma$ contains a point in the interior of $\mathbb{R}_{\geq 0}^E$, therefore we have $\dim(\sigma) = \dim(\H_{1}(\Upsilon, \mathbb{R}))$. The second equality in formula (\ref{dim1}) is trivial.
\end{proof}

Denote by $Y_1, \dots, Y_F$ the connected components of $\S_{0, 2g} \setminus \Upsilon$. We have $Y_i \cong \S_{0, k_i}$ for some $k_i$. Recall that by $\S_0^{k}$ we denote the sphere with $k$ boundary components.
\begin{prop} \label{mcg}
 $\Stab_{\PMod(\S_{0, 2g})}(\sigma) \cong \Mod(\S_0^{k_1}) \times \dots \times \Mod(\S_0^{k_F})$.
\end{prop}
\begin{proof}
	By construction we have $\Stab_{\PMod(\S_{0, 2g})}(\sigma) = \Stab_{\PMod(\S_{0, 2g})}(\Upsilon)$
	Denote by $\overline{Y}_i$ the closure of $Y_i = \S_{0, k_i}$ in the sphere. Let $\widetilde{Y}_i \cong \S_{0}^{k_i}$ be the compactification of $Y_i$ given by replacing each puncture by a boundary component. Let $p_i: \widetilde{Y}_i \to \overline{Y}_i$ be the natural projection. Then we have the corresponding mapping $\Phi_i: \Mod(\widetilde{Y}_i) \to \Stab_{\PMod(\S_{0, 2g})}(\Upsilon)$.
	It suffices to prove that the obvious mapping
	\begin{equation*}
	\Phi : \Mod(\widetilde{Y}_1) \times \dots \times \Mod(\widetilde{Y}_F) \to \Stab_{\PMod(\S_{0, 2g})}(\Upsilon)
	\end{equation*} 
	is an isomorphism. We use the Alexander method (see \cite[Proposition 2.8]{Primer}). In the proof we need to distinguish between mapping classes and their representatives. The mapping class of a homeomorphism $\psi$ is denoted by $[\psi]$. 
	
	First we prove the surjectivity of $\Phi$. Let $[\psi] \in \Stab_{\PMod(\S_{0, 2g})}(\Upsilon)$. Then $\psi(\delta)$ is isotopic to $\delta$ for each arc $\delta$ of $\Upsilon$. All such arcs are disjoint, so the Alexander method implies that there is a representative $\psi' \in [\psi]$ such that $\psi'(\delta) = \delta$ for each arc $\delta$ of $\Upsilon$. Denote $\phi'_i = \psi'|_{\overline{Y}_i}$. Since $\phi'_i$ is identical on $\partial \overline{Y}_i$ there exist $\phi_i \in \Homeo^+(\widetilde{Y}_i)$ such that $p_i \circ \phi_i = \phi'_i \circ p_i$. Hence we have $\Phi([\phi_{1}], \dots, [\phi_{F}]) = [\psi]$.
	
	Now we prove that $\Phi$ is injective. Let $\Phi ([\psi_{1}], \dots, [\psi_{F}]) = [\id]$. Since for each $i$ the mapping $\psi_{i}|_{\partial \widetilde{Y}_i}$ is identical, there exists $\psi'_i \in \Homeo^+(\overline{Y}_i)$ such that $p_i \circ \psi_i = \psi'_i \circ p_i$. Consider the mapping $\psi' \in \Homeo^+(\S_{0, 2g})$ such that $\psi'|_{\overline{Y}_i} = \psi'_i$ for all $i$. By assumption $\psi'$ is isotopic to the identity map.
	
	Let $\widetilde{\Upsilon}$ be a planar graph on the sphere obtained from $\Upsilon$ by adding several arcs such that each face of $\widetilde{\Upsilon}$ is a disk.
	Let us show that there is an isotopy $\Psi_t: \S_{0, 2g} \to \S_{0, 2g}$ with $\Psi_0 = \psi'$ such that $\Psi$ restricts to the identity on $\Upsilon$ and $\Psi_1(\psi'(\delta)) = \delta$ for each arc $\delta$ of $\widetilde{\Upsilon}$. It suffices to prove the existence of such isotopy in the case when we add only one arc $\gamma$ to $\Upsilon$. Let $\Upsilon' = \Upsilon \cup \{\gamma\}$. We can assume that $\psi'(\gamma)$ is transversal to $\gamma$. If $\psi'(\gamma)$ is disjoint from $\gamma$ then these two arcs bound a disk on $\S_{0, 2g}$. This disk is contains no punctures, so it is disjoint from $\Upsilon$. Hence in this case such un isotopy exists. If $\psi'(\gamma)$ and $\gamma$ intersect, they form a bigon (see \cite[Proposition 1.7]{Primer}) that is also disjoint from $\Upsilon$ by the same reason. Hence we can decrease the number of intersection points of $\gamma$ and $\psi'(\gamma)$.
	
	Denote $\phi' = \Psi_1$, $\phi'_i = \phi'|_{\overline{Y}_i}$ and $\Psi'_i = \Psi|_{\overline{Y}_i}$. There exist the homeomorphisms $\phi_i \in \Homeo(\widetilde{Y}_i)$ and the isotopies $\Psi_i$ of $\widetilde{Y}_i$ such that $p_i \circ \phi_i = \phi'_i \circ p_i$ and $p_i \circ \Psi_i = \Psi'_i \circ p_i$. Therefore $\Psi_i$ is an isotopy between $\psi_i$ and $\phi_i$. By construction $\phi_i$ is identical on a collection of arcs that fill $\widetilde{Y}_i$ (\textit{fill} mean that each connected component of the complement to this collection is a disk). Hence the Alexander method implies that $\phi_i$ is isotopic to identity for each $i$. Therefore $\psi_i$ is also isotopic to identity. This concludes the proof.
\end{proof}

For $k \geq 2$ we have $\Mod(\S_{0, k-1}^1) \cong \PB_{k-1}$. If we replace the punctures on the disk $S_0^1$ by boundary components, the correspondent mapping class groups will by related to each other via the following exact sequence (see \cite[Proposition 3.19]{Primer})
$$1 \to \Z^{k-1} \to \Mod(\S_0^{k}) \to \Mod(\S^1_{0, k-1}) \to 1.$$
Since the tangent bundle to the disk is trivial, this sequence splits.
Therefore we have
$\Mod(\S_0^{k}) \cong \Z^{k-1} \times \PB_{k-1}$. Since $\cd(\PB_{k-1}) = k-2$ then $\cd(\Z^{k-1} \times \PB_{k-1}) = 2k-3$. In the case $k=1$ we have $\cd(\Mod(\S_0^{1})) = 0$. Denote by $D$ the number of $Y_i$ that are homeomorphic to the disk. Proposition \ref{mcg} immediately implies the following result.
\begin{corollary} \label{cd}
	$\cd(\Stab_{\PMod(\S_{0, 2g})}(\sigma)) = \sum_{i=1}^F(2k_i - 3) + D.$
\end{corollary}

Let us finish the proof of Theorem \ref{stabdim}. By Lemma \ref{dim} and Corollary \ref{cd} we have
$$
\dim(\sigma) + \cd(\Stab_{\PMod(\S_{0, 2g})}(\sigma)) = E - V + C + \sum_{i=1}^F(2k_i - 3) + D = $$
$$= E-V+C+D-3F + 2\sum_{i=1}^F k_i.$$
Let $\Theta_1, \dots, \Theta_C$ be the connected components of $\Upsilon$.
Note that
\begin{equation} \label{eqq}
\sum_{i=1}^{F} k_{i} = \Bigl|\{ (Y_{i}, \Theta_{j}) \; | Y_{i} \mbox{ is adjacent to } \Theta_{j}\}\Bigr| = \sum_{j=1}^{C} (\dim(H_{1}(\Theta_j, \mathbb{R})) + 1) =
\end{equation}
$$= \dim(H_{1}(\Upsilon, \mathbb{R})) + C = E - V + 2C.$$
Therefore, we have
$$E-V+C+D-3F + 2\sum_{i=1}^F k_i = E-V+C+D-3F + 2(E - V + 2C) = $$
$$= 3E - 3V +5C - 3F + D = 2C + D -3(V-E+F-C).$$
By Euler's formula we have
\begin{equation}\label{fact}
V-E+F-C = 1.
\end{equation}
Therefore  
\begin{equation} \label{eq}
\dim(\sigma) + \cd(\Stab_{\PMod(\S_{0, 2g})}(\sigma)) \leq 2C+D-3.
\end{equation}

In order ti finish the proof of Theorem \ref{stabdim} we need the following result.
\begin{lemma}\label{disks}
	Let a planar graph $\Upsilon$ represent a cell of $\B_{0, 2g}$ and $g \geq 2$.
	Then $2C + D \leq 2g$.
\end{lemma}
\begin{proof}
	We prove the lemma by induction on the number of connected components of $\Upsilon$ with only one edge.
	
	\textit{Base case: $\Upsilon$ does not have a connected component with only one edge.}
	Since $D \leq F$ and $V = 2g$ it suffices to check that 
	\begin{equation}\label{ineq1}
	2C + F \leq V.
	\end{equation}
	
	Note that $\Upsilon$ is a bipartite graph and does not contain isotopic edges. Since $\Upsilon$ does not have a connected component with only one edge, if $Y_{i}$ adjacent to $\Theta_{j}$ for some $i$ and $j$, then $Y_{i}$ adjacent to at least $4$ edges of $\Theta_{j}$. Then by (\ref{eqq}) we have
	$$E \geq 2\sum_{i=1}^{F} k_{i} = 2E - 2V + 4C = 2C + 2F - 2.$$
	The last equality follows from (\ref{fact}). Since $C \geq 1$ we have
	$$E \geq 2C + 2F - 2 \geq 2F + C - 1.$$
	We can rewrite this as follows.
	\begin{equation} \label{ineq3}
	2C + F \leq 1 + C - F + E.
	\end{equation}
	Fact \ref{fact} implies that the right hand side of (\ref{ineq3}) equals $V$. Therefore inequality (\ref{ineq1}) holds.
	
	\textit{Induction step: $\Upsilon$ has a connected component with only one edge.} In the case $g = 2$ the graph $\Upsilon$ is disjoint union of two closed intervals, so we have $C=2$ and $D=0$; in this case the required inequality  
	$2C + D \leq 2g$
	is obvious. Hence we can assume that $g\geq 3$.
	Let $p_i$ and $q_j$ form such a component, that is, $p_i$ and $q_j$ are vertices of $\Upsilon$ of degree one connected by an edge $\alpha$. Assume that after removing this component the remaining graph $\Upsilon_1$ will not contain isotopic edges (and, consequently, will represent some cell of $\B_{0, 2g-2}$).
	Then $C_1=C-1$ is the number of connected components of $\Upsilon_1$. Denote by $D_1$ the number of faces of $\Upsilon_1$ homeomorphic to the disk. We have $D_1 \leq D+1$, since at most one disk can appear. The graph $\Upsilon_1$ has less connected components with only one edge than $\Upsilon$. Since $g-1 \geq 2$, by the induction assumption we have
	$$2C + D \leq 2C_1 + D_1 + 1 \leq 2g-2+1 < 2g.$$
	 
	 Now assume that our previous assumption does not hold, that is, after removing the component consisting of one edge the remaining graph will contain isotopic edges. This means that there exist punctures $p_r$, $q_s$ and edges $\beta_1$, $\beta_2$ between them, such that $p_i$ and $q_j$ are the only vertices of $\Upsilon$ located inside of the disks bounded by $\beta_1$ and $\beta_2$.
	 There exist an arc $\gamma_1$ from $p_i$ to $q_s$ and an arc $\gamma_2$ from $p_r$ to $q_j$, such that $\gamma_1$ and $\gamma_2$ are disjoint from $\Upsilon$ and from each other. Consider the graph $\Upsilon'$ obtained from $\Upsilon$ by adding the edges $\gamma_1$ and $\gamma_2$. Note that $\Upsilon'$ has less connected components with exactly one edge than $\Upsilon$ and also represents a cell of $\B_{0, 2g}$.
	 Then $C' = C-1$ is the number of connected components of $\Upsilon'$ and $D' = D+2$ is the number of faces of $\Upsilon'$ homeomorphic to the disk. Therefore we have $2C+D \leq 2g$ if and only if $2C' + D' \leq 2g$. The induction assumption concludes the proof.
\end{proof}

Lemma \ref{disks} and inequality (\ref{eq}) imply that 
$$\dim(\sigma) + \cd(\Stab_{\PMod(\S_{0, 2g})}(\sigma)) \leq 2g-3.$$
This completes the proof of Theorem \ref{stabdim}.
\end{proof}

\subsection{The spectral sequence}

Let $K \subseteq \PMod(\S_{0, 2g})$ be a subgroup. Denote by $\widehat{E}_{*, *}^{*}$ the spectral sequence (\ref{spec_sec}) for the action of $K$ on $\B_{0, 2g}$. Since cohomological dimension is monotonous, Theorem \ref{stabdim} implies that for any cell $\sigma$ of $\B_{0, 2g}$ we have 
\begin{equation*} \label{cdineq}
\dim(\sigma) + \cd(\Stab_{K}(\sigma)) \leq 2g-3.
\end{equation*}
This immediately implies $\widehat{E}^1_{p, q} = 0$ for $p+q > 2g-3$. Hence all differentials $d^1, d^2, \dots$ to the group $\widehat{E}^1_{0, 2g-3}$ are trivial (Fig. \ref{spectral} is also applicable here, the group $\widehat{E}_{0, 2g-3}^1$ is shown in green), so $\widehat{E}^1_{0, 2g-3} = \widehat{E}^{\infty}_{0, 2g-3}$. Therefore we have the following result.
\begin{prop} \label{prop2}
	Let $\mathfrak{L} \subseteq \mathcal{L}_{0}$ be a subset consisting of multiarcs from pairwise different $K$-orbits. Then the inclusions $\Stab_{K}(L) \subseteq K$, $L \in \mathfrak{L}$ induce the injective homomorhpism
	\begin{equation*} \label{c-l3}
	\bigoplus_{L \in \mathfrak{L}} \H_{2g-3}(\Stab_{K}(L), \Z) \hookrightarrow \H_{2g-3}(K), \Z).
	\end{equation*}
\end{prop}

\begin{proof}[Proof of Lemma \ref{4lemma}] It suffices to prove that the inclusions 
\begin{equation*} \label{inclu2}
j_u: \Stab_{\K_g}(\hat{u} \cdot N) \hookrightarrow \K_g, \; \; \; u \in \U_g
\end{equation*}
induce the injective homomorphism
\begin{equation*} \label{eq9}
\bigoplus_{u\in \U_g} \H_{2g-3}(\Stab_{\K_g}(\hat{u} \cdot N), \Z) \hookrightarrow G_1 \subseteq \H_{2g-3}(\K_g, \Z).
\end{equation*}

Assume the converse and consider a nontrivial linear relation
\begin{equation} \label{relation2}
\sum_{s=1}^k \lambda_s (j_{u_s})_*(\theta_s) = 0, \; \; \lambda_s \in \Z, \; \theta_s \in \H_{2g-3}(\Stab_{\K_g}(\hat{u}_s \cdot N), \Z)
\end{equation}
for some pairwise different $u_1, \dots, u_s \in \U_g$. 
For each $i = 1, \dots, g$ take an essential simple closed curve $\beta_{i} = \beta_{1, i}$ on the one-punctured torus $X_i$. Denote by $b_i = [\beta_{1, i}] \in H_{1}(\S_{g}, \Z)$ the corresponding homology class. For each $s \in \N$ denote by $\hat{X}_{s, i} \subset \S_{g}$ the one-punctured torus bounded by $\hat{u}_s \cdot \delta_{i}$. Since $\hat{u}_s$ belongs to the Torelli group $\I_g$ we have $\H_{1}(\hat{X}_{s, i}, \Z) = \H_{1}(\hat{X}_{t, i}, \Z)$ for all $1 \leq s, t \leq k$. Denote by $\beta_{s, i}$ a unique curve on $\hat{X}_{s, i}$ representing the homology class $b_i$.

Let $B_s = \beta_{s, 1} \cup \dots \cup \beta_{s, g}$. Let $\{B_{d_1}, \dots, B_{d_l}\} \subseteq \{B_1, \dots, B_k\}$ be the maximal subset consisting of the multicurves from pairwise distinct $\K_g$-orbits.
Take the homology class $x = \sum_{i=1}^g b_i$ and consider the complex of cycles $\B_g(x)$.
Proposition \ref{prop1} implies that the inclusions
\begin{equation*}
\iota_i: \Stab_{\K_g}(B_{d_i}) \hookrightarrow \K_g
\end{equation*}
induce the injective homomorhpism
\begin{equation} \label{inc3}
\bigoplus_{i=1}^l \H_{2g-3}(\Stab_{\K_g}(B_{d_i}), \Z) \hookrightarrow \H_{2g-3}(\K_g, \Z).
\end{equation}
Since the curves $\beta_{s, i}$ can be chosen in a unique way, we have the inclusions
\begin{equation*}
\hat{j}_{u_s}: \Stab_{\K_g}(\hat{u}_s \cdot N) \hookrightarrow \Stab_{\K_g}(B_s).
\end{equation*}
Since $j_{u_{i}} = \iota_{i} \circ \hat{j}_{u_{i}}$ (\ref{inc3}) and (\ref{relation2}) imply that for each $i = 1, \dots, l$ we have
\begin{equation} \label{relation3}
\sum_{\{z \; | \; B_z \in \Orb_{\K_g}(B_{d_i})\}} \lambda_z (j_{u_z})_*(\theta_z) = 0.
\end{equation}

Equality (\ref{relation3}) implies that it is sufficient to prove the statement of the lemma in case where the multicurves $B_1, \dots, B_k$ belong to the same $\K_g$-orbit. Since we can prove Lemma \ref{4lemma} of an arbitrary choice of $\hat{U}$, then by choosing the lifts $\hat{u}$ we can assume that $B_1 = \dots = B_k = B$. Let $\zeta_{s, i}$ be a curve on $\hat{X}_{s, i}$ intersecting $\beta_i$ once and let $L_s = \zeta_{s, 1} \cup \dots \cup \zeta_{s, g}$. Consider the surface $\S_{g} \setminus B \cong \S_{0, 2g}$. Denote by $p_i$ and $q_i$ the punctures on $\S_{0, 2g}$ corresponding to the two sides of the curve $\beta_i$.

Consider the exacts sequence (\ref{LB}) in the case $M = B$. We have
\begin{equation*} \label{LB2}
1 \rightarrow \left\langle T_{\beta_{1}}, \dots, T_{\beta_{g}} \right\rangle  \rightarrow \Stab_{\Mod(\S_{g})}(B) \rightarrow \Mod(\S_{0, 2g}) \rightarrow 1.
\end{equation*}
Since the intersection $\left\langle T_{\beta_{1}}, \dots, T_{\beta_{g}} \right\rangle \cap \K_g$ is trivial we have the inclusion $K = \Stab_{\K_g}(B) \hookrightarrow \Mod(\S_{0, 2g})$. The action of $\K_g$ on the homology of $\S_g$ is trivial, so the image of this inclusion is contained in $\PMod(\S_{0, 2g})$. We have $K \hookrightarrow \PMod(\S_{0, 2g})$. Denote by $\zeta'_{s, i}$ the arc on $\S_{0, 2g}$ from $p_i$ to $q_i$ corresponding to the curve $\zeta_{s, i}$ and let $L'_s = \zeta'_{s, 1} \cup \dots \cup \zeta'_{s, g}$. Let us show that $L'_1, \dots, L'_k$ belong to pairwise distinct $K$-orbits. 

Assume the converse, then $f(L'_1) = L'_2$ for some $f \in K$. Then $f(L_1 \cup B) = L_2 \cup B$. Note that the surface $\S_g \setminus (L_s \cup B)$ has $g$ punctures and each component of $\hat{u}_s \cdot N$ is homotopic into a neighborhood of its own puncture.
Therefore the correspondent components of the multicurves $f(\hat{u}_1 \cdot N)$ and $\hat{u}_2 \cdot N$ are homotopic into a neighborhood of the same puncture. Consequently, we have that the multicurves $f(\hat{u}_1 \cdot N)$ and $\hat{u}_2 \cdot N$ are isotopic. Since $\hat{u}_1, \hat{u}_2 \in \I_g$, we obtain $\hat{u}_2^{-1} f \hat{u}_1 \in \Stab_{\I_g}(N)$. It follows from the exactness of (\ref{exact}) that $\Stab_{\I_g}(N) \subseteq \K_g$. Hence we have $\hat{u}_2^{-1} f \hat{u}_1 \in \K_g$ and we obtain 
$$0 = \tau(\hat{u}_2^{-1} f \hat{u}_1) = \tau(\hat{u}_1) - \tau(\hat{u}_2),$$
where $\tau$ is the Johnson homomorphism. This implies $u_1 = u_2$; this contradiction proves.

Therefore $L'_1, \dots, L'_k$ belong to pairwise distinct $K$-orbits. Proposition \ref{prop2} implies that the inclusions $\Stab_{K}(L'_s) \subseteq K$, $L' \in \mathfrak{L}$ induce the injective homomorhpism
\begin{equation*} \label{c-l4}
\bigoplus_{s} \H_{2g-3}(\Stab_{K}(L'_s), \Z) \hookrightarrow \H_{2g-3}(K, \Z).
\end{equation*}
By Proposition \ref{prop1} we also have the inclusion $\H_{2g-3}(K, \Z) \hookrightarrow \H_{2g-3}(\K_g, \Z)$.
We have $\Stab_{K}(L'_s) = \Stab_{\K_g}(\hat{u}_s \cdot N)$. Therefore (\ref{relation2}) implies $\lambda_s = 0$ for all $s$. This contradiction proves Lemma \ref{4lemma}.
\end{proof}

\end{document}